\documentclass[11pt]{article}
\usepackage{amsmath}
\usepackage{amsfonts}
\usepackage{color}
\usepackage{latexsym}
\usepackage{tablefootnote}
\usepackage{epsfig}
\usepackage{multirow}
\usepackage{float}
\usepackage{array}
\usepackage{bbm}
\allowdisplaybreaks

\usepackage{makecell,booktabs}
\usepackage[version=4]{mhchem}
\usepackage[strict]{changepage}

\usepackage{amsmath}
\usepackage{amsfonts}
\usepackage{latexsym}
\usepackage{amssymb}

\newtheorem{thm}{Theorem}[section]
\newtheorem{theorem}[thm]{Theorem}

\newtheorem{lemma}[thm]{Lemma}

\newtheorem{proposition}[thm]{Proposition}

\newtheorem{remark}[thm]{Remark}

\newcommand{\beq}{\begin{equation}}
\newcommand{\eeq}{\end{equation}}
\newcommand{\beqa}{\begin{eqnarray}}
\newcommand{\eeqa}{\end{eqnarray}}
\newcommand{\beqas}{\begin{eqnarray*}}
\newcommand{\eeqas}{\end{eqnarray*}}
\newcommand{\bi}{\begin{itemize}}
\newcommand{\ei}{\end{itemize}}

\newcommand{\nn}{\nonumber}

\newcommand{\R}{\mathbb{R}}

\newcommand{\lam}{{\lambda}}

\newcommand{\inner}[2]{\langle #1,#2\rangle}

\newcommand{\argmin}{\mathrm{argmin}\,}

\newcommand{\dom}{\mathrm{dom}\,}

\newcommand{\bConv}[1]{\overline{\mbox{\rm Conv}}\,(\R^{#1})}

\newcommand{\tx}{\tilde x}
\newcommand{\ty}{\tilde y}

\begin{document}
	\title{A FISTA-type accelerated gradient algorithm for solving  \\ smooth nonconvex composite optimization problems}
	\date{May 16, 2019 (1st revision: November 4, 2019; 2nd revision: March 5, 2021)}
	\author{
		Jiaming Liang \thanks{School of Industrial and Systems
			Engineering, Georgia Institute of
			Technology, Atlanta, GA, 30332-0205.
			(email: {\tt jiaming.liang@gatech.edu} and {\tt renato.monteiro@isye.gatech.edu}). This work
			was partially supported by ONR Grant N00014-18-1-2077.}\qquad 
		Renato D.C. Monteiro \footnotemark[1]\qquad 
		Chee-Khian Sim \thanks{School of Mathematics and Physics, University of Portsmouth, Lion Gate Building, Lion Terrace, Portsmouth PO1 3HF. (email: {\tt chee-khian.sim@port.ac.uk}). This work is made possible through an LMS Research in Pairs (Scheme 4) grant.} 
	}
	\maketitle
	
	%
	%
	
	\begin{abstract}
		In this paper, we describe and establish iteration-complexity of two accelerated composite gradient (ACG) variants to solve a smooth nonconvex composite optimization problem whose objective function is the sum of a nonconvex differentiable function $ f $ with a Lipschitz continuous gradient and a simple nonsmooth closed convex function $ h $. When $f$ is convex, the first ACG variant reduces to the well-known FISTA for a specific choice of the input, and hence the first one can be viewed as a natural extension of the latter one to the nonconvex setting. The first variant requires an input pair $(M,m)$ such that
		$f$ is $m$-weakly convex, $\nabla f$ is $M$-Lipschitz continuous, and $m \le M$ (possibly $m<M$), which is usually hard to obtain or poorly estimated. The second variant on the other hand can start from an arbitrary input pair $(M,m)$ of positive scalars and its complexity is shown to be not worse, and better in some cases, than that of the first variant for a large range of the input pairs. Finally, numerical results are provided to illustrate the efficiency of the two ACG variants.
	\end{abstract}

	\section{Introduction}

	Accelerated gradient methods for solving convex noncomposite  programs were originally developed by
	Nesterov in his celebrated work \cite{ag_nesterov83}. Subsequently, several variants of this method (see for example \cite{beck2009fast,lan2011primal,MonteiroSvaiterAcceleration,nesterov2012gradient,Nest05-1,tseng2008accmet})  were developed for solving
	convex simple-constrained or composite programs, which we refer generically to as ACG variants.
	These variants have also been used as subroutines in several inexact-type proximal algorithms
	for solving convex-concave saddle point and monotone Nash equilibrium problems (see for example \cite{chen2014optimal, He3, He1, Kolossoski, Nest05-1,LanADMM}).
	
	In this paper, we study
	ACG algorithms to solve the smooth nonconvex composite optimization (SNCO) problem
	\begin{equation}\label{eq:probintro}
		\phi_*:=\min \left\{ \phi(z):=f(z) + h(z)  : z \in \mathbb{R}^n \right \}
	\end{equation}
	where $h:\R^n \to (-\infty,\infty]$ is a proper lower-semicontinuous convex function with bounded $\dom h$
	and $f$ is a real-valued differentiable (possibly nonconvex) function whose gradient is $M$-Lipschitz continuous
	on $\dom h$, i.e., for every $z, z' \in \dom h$,
	\begin{equation}\label{ineq:Lips}
		\| \nabla f(z') - \nabla f(z) \| \le M \|z'-z\| .
	\end{equation}
	The first analysis of an ACG algorithm for solving \eqref{eq:probintro} under the above assumption appears in \cite{nonconv_lan16}
	where essentially a well-known ACG variant that solves the convex version of \eqref{eq:probintro} is also shown to
	solve its nonconvex version in the following sense: for a given tolerance $\hat \rho>0$, it computes $(\hat y,\hat v)  \in \dom h \times \R^n$ such that
	$\hat v \in  \nabla f(\hat y) + \partial h(\hat y)$ and $\| \hat{v} \| \leq \hat{\rho}$ in
	\begin{eqnarray}\label{complexityresult}
		\mathcal{O} \left( \frac{{M}{\bar m}D_h^2}{\hat{\rho}^2} + \left(\frac{Md_0}{\hat{\rho}}\right)^{2/3} \right)
	\end{eqnarray}
	iterations where $d_0$ is the distance of the initial point $x_0$ to the optimal solution set of \eqref{eq:probintro},
	$D_h$ is the diameter of $\dom h$ and ${\bar m}$ is the smallest scalar
	$m \ge 0$ such that
	\begin{equation}\label{ineq:curv}
		-\frac m2\|z'-z\|^2\le f(z')- f(z) - \inner{\nabla f(z)}{z'-z}.
	\end{equation}
	for every $z, z' \in \dom h$.
	Any pair $ (M,m) $ with $m \le M$ and
	satisfying both \eqref{ineq:Lips} and \eqref{ineq:curv} is referred to as a curvature pair.
	We refer to the ACG variant of \cite{nonconv_lan16} as the AG method and note that
	each one of its iterations performs exactly two resolvent evaluations of $h$, i.e., an
	evaluation of the point-to-point operator $(I+ \tau \partial h)^{-1}(\cdot)$ for some $\tau>0$.
	(Several examples of convex, as well as nonconvex,
	functions $h$ whose resolvent evaluations are easy to compute
	can be found in \cite{gong2013general}.)

	%
	This paper describes and establishes the iteration-complexities of two
	ACG variants for solving the nonconvex version of \eqref{eq:probintro}.
	The first variant can be viewed as a direct extension of the FISTA presented
	in  \cite{beck2009fast} for solving the convex version of \eqref{eq:probintro}.
	In contrast to an iteration of the AG method,
	every iteration of the first variant performs exactly 
	one resolvent evaluation of $h$. One drawback of the first variant is that it requires as input a curvature pair $(M,m)$, 
	which is usually hard to obtain or is poorly estimated.
	Letting $(\bar M,\bar m)$ denote the smallest curvature pair,
	a second variant is proposed to remedy the aforementioned drawback in that it works regardless of the choice of input pair $(M,m)$
	(i.e.,
	not necessarily satisfying \eqref{ineq:Lips} and \eqref{ineq:curv}),
	and its complexity is shown to be not worse  than \eqref{complexityresult} when $M \ge \bar M$ and $m \in [\bar m, M]$.
	Moreover, when $m \in [\bar m, \bar M]$, the complexity of the second variant is
	empirically argued to behave as \eqref{complexityresult} with $M=\bar M$,
	for a large range of scalars $M$  such that $M \le \bar M$
	(see the second paragraph following Theorem \ref{thm:ADAP-NC-FISTA}) and our computational results demonstrate that
	taking $M$ relatively smaller than $\bar M$ can substantially improve its performance.
	It is also shown that all iterations of the second variant,
	with the exception of a few ones whose total number is $\log$-bounded,
	perform exactly one resolvent evaluation of $h$.
	

	{\bf Related works.}  Inspired by \cite{nonconv_lan16}, other papers have proposed ACG variants for solving (\ref{eq:probintro}) under the assumption that $f$ is a nonconvex continuously differentiable function
	with a Lipschitz continuous gradient, and that $h$ is a simple lower semi-continuous
	convex (see e.g.\ \cite{Paquette2017, LanUniformly}) or nonconvex (see e.g.\ \cite{Li_Lin2015, li2017convergence, Yao_et.al.}) function. 
	Similar to an iteration of the two ACG variants in our paper, the one of the algorithms in \cite{li2017convergence,Yao_et.al.} requires exactly one resolvent evaluation of $h$.
	However, while every iteration of the variants studied here is always accelerated,
	the ones of the latter algorithms can be a simple composite gradient (and unaccelerated)
	step whenever a certain descent property is not satisfied.
	
	Another approach for solving \eqref{eq:probintro} consists of using a descent unaccelerated inexact
	proximal-type method 
	where each prox subproblem is constructed to be (possibly strongly) convex and hence solved by an ACG variant
	(see \cite{carmon2018accelerated,KongMeloMonteiro,rohtua}).
	Moreover, the approach has the benefit of working with a larger prox stepsize and hence of having a better outer iteration-complexity than the approaches in the previous paragraph.
	However, each of its outer iterations still has to perform a uniformly bounded
	number of inner iterations to approximately solve a prox subproblem.
	Overall, it is shown that its inner-iteration complexity is better than the iteration-complexities of
	the methods in the previous
	paragraph, particularly when $\bar m \ll \bar M$.
	As in the papers \cite{Paquette2017, LanUniformly,Li_Lin2015, li2017convergence, Yao_et.al.} in the previous paragraph, it is worth noting that
	the method in \cite{rohtua} attempts to
	perform an accelerated step whenever a certain descent property holds and, in case of failure,
	it performs an unaccelerated prox step similar to the one used in the methods in \cite{carmon2018accelerated,KongMeloMonteiro}.
	
	Finally, a hybrid approach that borrows ideas from the above group of papers is presented in \cite{jliang2018double}. More specifically, the latter work presents
	an accelerated inexact proximal point method reminiscent of those presented
	in \cite{guler1992new,MonteiroSvaiterAcceleration,salzo2012inexact}, but in which only the convex version of \eqref{eq:probintro} is considered.
	Each (outer) iteration of the method requires that a prox subproblem be approximately solved
	by using an ACG variant in the same way as in the papers \cite{carmon2018accelerated,KongMeloMonteiro}.
	Hence, similar to the methods in the previous paragraph,
	this method performs both outer and inner iterations with a major difference
	that every outer iteration is an accelerated step (as in the papers \cite{Paquette2017, LanUniformly,Li_Lin2015, li2017convergence, Yao_et.al.}) with a large
	proximal stepsize (as in the papers \cite{carmon2018accelerated,KongMeloMonteiro}).

	{\bf Organization of the paper.}  Subsection~\ref{sec:DefNot} presents basic definitions and notations used throughout the paper. 
	Section~\ref{sec:alg+cmplx} presents assumptions made on the SNCO problem, describes the first  ACG variant, which is an extension of FISTA to the SNCO problem and is referred to as NC-FISTA, and establishes its iteration-complexity for obtaining a stationary point of the SNCO problem. 
	Section~\ref{variants} presents an adaptive variant of NC-FISTA, namely,  ADAP-NC-FISTA, and establishes its iteration-complexity.
	Section~\ref{sec:computResults} presents computational results showing the efficiency of
	NC-FISTA and ADAP-NC-FISTA.
	Section~\ref{sec:conclusions} finishes the paper by presenting a few concluding remarks.
	Finally, supplementary technical results are provided in the appendix.
	\subsection{Basic definitions and notation} \label{sec:DefNot}
	
	This subsection provides some basic definitions and notations used
	in this paper.
	
	
	The set of real numbers is denoted by $\mathbb{R}$. The set of non-negative real numbers  and 
	the set of positive real numbers are denoted by $\mathbb{R}_+$ and $\mathbb{R}_{++}$, respectively. Let $\mathbb{R}^n$ denote the standard $n$-dimensional Euclidean 
	space with  inner product and norm denoted by $\left\langle \cdot,\cdot\right\rangle $
	and $\|\cdot\|$, respectively. 
	The Frobenius inner product and Frobenius norm in $ \R^{m\times n} $
	are denoted by $ \inner{\cdot}{\cdot}_F $ and $ \|\cdot\|_F $, respectively.
	The sets of real $ n \times n $ symmetric positive semidefinite matrices are denoted by $ S_{+}^n $.
	Let $ N_X(z) $ denote the normal cone of $ X $ at $ z $, i.e., $ N_X(z)=\{u\in \R^n: \inner{u}{z'-z} \le 0 \quad  \forall z' \in X \} $.
	The indicator function $I_X$ of a set
	$ X\subset \R^n $ is defined as $ I_X (z) =0 $ for every $ z\in X $, and
	$ I_X (z) =\infty $, otherwise. 
	If $ \Omega $ is a nonempty closed convex set, the orthogonal projection $ P_{\Omega}: \R^n \rightarrow \R^n $ onto $ \Omega $ is defined as 
	\[
	P_{\Omega}(z):=\argmin_{z'\in \Omega} \|z'-z\| \quad \forall z\in \R^n.
	\]
	Define $\log^+(s):= \max\{\log s ,0\}$ and $\log^+_1(s):= \max\{\log s ,1\}$ for $s>0$.
	
	Let $\Psi: \mathbb{R}^n\rightarrow (-\infty,+\infty]$ be given. The effective domain of $\Psi$ is denoted by
	$\dom \Psi:=\{x \in \mathbb{R}^n: \psi (x) <\infty\}$ and $\Psi$ is proper if $\dom \Psi \ne \emptyset$.
	Moreover, a proper function $\Psi: \mathbb{R}^n\rightarrow (-\infty,+\infty]$ is $\mu$-strongly convex for some $\mu \ge 0$ if
	$$
	\Psi(\beta z+(1-\beta) z')\leq \beta \Psi(z)+(1-\beta)\Psi(z') - \frac{\beta(1-\beta) \mu}{2}\|z-z'\|^2
	$$
	for every $z, z' \in \dom \Psi$ and $\beta \in [0,1]$.
	Let $\partial \Psi (z)$ denote the subdifferential of $\Psi$ at $z \in \dom \Psi$.
	If $\Psi$ is differentiable at $\bar z \in \mathbb{R}^n$, then its affine approximation $\ell_\Psi(\cdot;\bar z)$ at $\bar z$ is defined as
	\[
	\ell_\Psi(z;\bar z) :=  \Psi(\bar z) + \inner{\nabla \Psi(\bar z)}{z-\bar z} \quad \forall  z \in \mathbb{R}^n.
	\]
	Let $\bConv{n}$ denote the set of all proper lower semi-continuous convex functions $\Psi:\mathbb{R}^n\rightarrow (-\infty,+\infty]$.
	
	%
	
	\section{NC-FISTA for solving the SNCO problem}\label{sec:alg+cmplx}
	
	This section describes the assumptions made on our problem of interest, namely, problem \eqref{eq:probintro}.
	It also presents and establishes the iteration-complexity of 
	the first ACG variant, namely NC-FISTA,
	for obtaining an approximate solution of \eqref{eq:probintro}.
	
	Throughout this paper, we consider problem \eqref{eq:probintro} and make the following assumptions on it:
	\begin{itemize}
		\item[(A1)] $h \in \bConv{n}$;  
		\item[(A2)] $\dom h$ is bounded;
		\item[(A3)] $f$ is differentiable on a closed convex set $ \Omega \supseteq \dom h$ and there exists
		$M > 0$ such that \eqref{ineq:Lips} holds for every $z,z' \in \Omega$;
		\item[(A4)] $f$ is nonconvex on $\dom h$ and there exists $ m>0 $ such that \eqref{ineq:curv}
		holds for every $z,z' \in \Omega$.
	\end{itemize}
	
	Throughout this paper, we denote the diameter of $\dom h$ as
	\beq \label{eq:diam}
	D_h  :=  \sup \{  \| u'-u \| : u,u' \in \dom h \} < \infty
	\eeq
	where its finiteness is due to (A2).
	Moreover,  let  $\bar M$ (resp., $\bar m$) denote the smallest scalar $M$ (resp., $m$) satisfying \eqref{ineq:Lips}
	(resp., \eqref{ineq:curv})
	for every $z,z' \in \Omega$. Clearly, $\bar M \ge \bar m>0$.
	
	
	We now make a few remarks about the above assumptions.
	First, (A1)-(A3) imply that the set $Z^*$ of optimal solutions of \eqref{eq:probintro}
	is nonempty and compact.
	Second,
	using the fact that $\bar M$ satisfies \eqref{ineq:Lips} for every $z, z' \in \Omega$ in view of the above definition of
	$\bar M$, we easily see that
	\[
	\left| f(z')- \ell_f(z';z) \right| \le \frac {\bar M}2\|z'-z\|^2
	\quad \forall z, z' \in \Omega,
	\]
	and hence that \eqref{ineq:curv} is satisfied with $m=\bar M$. Thus, it follows that
	from the definition of $\bar m$ 
	that $\bar m\le \bar M$.
	Third, (A4) implies that $\bar m >0$.
	Fourth, our interest is in the case where
	$\bar m \ll \bar M$ since this case naturally arises in the context of penalty methods for solving
	linearly constrained composite nonconvex optimization problems (e.g., see  Section 4 of \cite{KongMeloMonteiro}).
	
	For $ z\in\dom  h$ to be a local minimizer of \eqref{eq:probintro}, a necessary condition is that
	$ z$ is a stationary point of \eqref{eq:probintro}, i.e.,  $0 \in \nabla f( z)+\partial h( z)$.
	Motivated by this remark, the following notion of an approximate solution to problem \eqref{eq:probintro} is proposed:
	a pair $( \hat{y},\hat{v})$ is said to be a $\hat\rho$-approximate solution to \eqref{eq:probintro}, for a given tolerance $\hat \rho>0$, if
	\begin{equation}\label{rho-sol}
		\hat {v} \in \nabla f(\hat{y}) + \partial h (\hat{y}), \quad \|\hat{v}\| \le  \hat \rho.
	\end{equation}

	We are now ready to state the NC-FISTA for solving  \eqref{eq:probintro}.
	
	\noindent\rule[0.5ex]{1\columnwidth}{1pt}
	
	NC-FISTA
	
	\noindent\rule[0.5ex]{1\columnwidth}{1pt}
	
	\begin{itemize}
		%
		\item[0.] Let an initial point $y_0 \in \dom h$, a pair $(M,m) \in \R^2_{++}$ such that $M \ge m \ge \bar m$ and
		$M > \bar M$, a scalar $A_0>0$, and
		a tolerance $\hat \rho>0$
		%
		%
		%
		%
		be given, and
		set $x_0=y_0$, $\lam=1/M$, $k=0$ and
		\begin{equation}\label{eq:kappa}
			\kappa_0 = \frac{ 1+ \sqrt{1 + 4 A_0}}{\sqrt{1 + 4 A_0}-1};
		\end{equation}
		\item[1.] compute
		\begin{align}\label{ak} 
			a_k = \frac{1+ \sqrt{1 + 4 A_k}}{2}, \quad A_{k+1} = A_k + a_k;
		\end{align}
		\item[2.] compute
		\begin{align}
			& \tx_k  =  \frac{A_k}{A_{k+1}}y_k + \frac{a_k}{A_{k+1}}x_k \label{tildexk} \\
			& {y}_{k+1} = {\mbox{argmin}}_u \left\{ \ell_f(u; \tilde{x}_k) + h(u) + \frac{1}{2}\left( \frac{1}{\lam}+\frac{\kappa_0 m}{a_k}\right)  \| u - \tilde{x}_k \|^2 \right\}, \label{minimize1} \\
			&  \hat x_{k+1} =  \frac{(a_k + \kappa_0 m\lam)y_{k+1}-(a_k-1)y_k}{\kappa_0 m \lam + 1}, \quad x_{k+1}= P_\Omega\left(\hat x_{k+1}\right) ; \label{minimize2prime}
		\end{align}
		\item[3.] compute 
		\begin{eqnarray}\label{def:v}
			v_{k+1} =  \left( \frac{1}{\lam}+\frac{\kappa_0 m }{a_k}\right) (\tilde{x}_k - y_{k+1}) + \nabla f(y_{k+1}) - \nabla f(\tilde{x}_k); \label{vk}
		\end{eqnarray}
		if $\|v_{k+1} \| \le \hat \rho$ then output  $(\hat y,\hat v)=(y_{k+1},v_{k+1})$ and {\bf stop};
		otherwise, set $k \leftarrow k+1$ and go to step 1.
	\end{itemize}
	\rule[0.5ex]{1\columnwidth}{1pt}
	
	We now make a few remarks about the NC-FISTA.
	First,  it follows from \eqref{minimize1} that $\{y_k\} \subset  \dom h$, and hence 
	$\{y_k\}$ is bounded in view of (A2).
	Second, the definition of $\{x_k\}$ in \eqref{minimize2prime} implies that $ \{x_{k}\}\subset \Omega $, and hence that
	$\{\tx_{k}\}\subset \Omega $ in view of \eqref{tildexk}. Hence, if $\Omega$ is chosen to be compact, then the latter two sequences
	will also be bounded but our analysis does not make such an assumption on $\Omega$.
	Third, if $\Omega = \R^n$, then each iteration of the NC-FISTA requires one resolvent evaluation of $ h$
	in \eqref{minimize1}, i.e.,
	an evaluation of $(I+ \tau \partial h)^{-1}$ for some $\tau>0$.
	Otherwise, it requires an extra projection onto $\Omega$ in \eqref{minimize2prime} , which, depending on the problem instance
	and the set $\Omega$,
	might be considerably cheaper than a resolvent evaluation of $ h$.
	Fourth, it follows from \eqref{ak} that $\{a_k\}$ and $\{A_k\}$ are strictly increasing sequences of positive scalars.
	Fifth, $A_0$ is required to be positive so as to guarantee that the quantity $\kappa_0$ defined in \eqref{eq:kappa} is well-defined.
	We will assume later on that $A_0={\Theta}(1)$ so as to eliminate it from the iteration-complexity bounds for NC-FISTA.
	Sixth, NC-FISTA requires that $M$ and $m$ be upper bounds for $\bar M$ and $\bar m$, respectively,
	due to technical requirements that appear in its iteration-complexity analysis.
	Actually, $M$ is also required to be not too close to $\bar M$.
	Seventh, if a scalar $M$ is known, then setting $m$ to be equal to $M$
	fulfills the conditions of step 0 of NC-FISTA in view of the fact that $\bar M \ge \bar m$.
	However, NC-FISTA also allows for the possibility that a sharper scalar $m \in [\bar m,M)$ is known due to
	the fact that  its iteration-complexity bound improves as $m$ decreases (see Theorem \ref{thm:main}).
	Eighth, when $ f $ is convex, i.e., $ \bar m=0 $, NC-FISTA reduces to FISTA if  $m$ is set to zero.
	Finally, \eqref{ak} implies that
	\begin{eqnarray}\label{Ak+1ak}
		A_{k+1} = a_k^2.
	\end{eqnarray}
	
	We establish a number of technical results. The first one establishes an important inequality satisfied by  $ m $.
	
	\begin{lemma}\label{taukproperty}
		For $ k\ge 0 $, we have
		\[
		\frac{\bar m}{\kappa_0} + \frac{ m}{a_k} \leq  m.
		\]
	\end{lemma}
	\begin{proof}
		Using the assumption $ m\ge \bar m $, the definition of $ \kappa_0 $ in \eqref{eq:kappa},  relation \eqref{ak} with $k=0$, and the fact that $\{a_k\}$ is increasing, we conclude that for every $ k\ge 0 $,
		\[
		m-\frac{\bar m}{\kappa_0} \ge  \left( 1-\frac{1}{\kappa_0}\right)  m=\frac{2 m}{1+\sqrt{1+4A_0}}=\frac{m}{a_0}\ge \frac{m}{a_k}.
		\]
	\end{proof}
	
	The following results introduce two functions that play important roles in our analysis of NC-FISTA and establish some basic facts about them.
	
	\begin{lemma}\label{gammatildegamma}
		For every $k \ge 0$, if we define
		\begin{align}
			& \tilde{\gamma}_k(u) :=  \ell_f(u; \tilde{x}_k) + h(u) + \frac{\kappa_0 m}{2a_k} \| u - \tilde{x}_k \|^2, \label{def1} \\
			& \gamma_k(u) := \tilde{\gamma}_k(y_{k+1}) + \frac{1}{\lambda}\langle \tilde{x}_k - y_{k+1}, u - y_{k+1} \rangle + \frac{\kappa_0 m}{2a_k} \| u - y_{k+1} \|^2,  \label{def2}
		\end{align}
		then the following statements hold:
		\begin{itemize}
			\item[(a)] both $\gamma_k$ and $\tilde \gamma_k$ are $(\kappa_0 m/a_k)$-strongly convex functions, $\gamma_k$ minorizes $\tilde{\gamma}_k$, $\tilde{\gamma}_k(y_{k+1}) = \gamma_k(y_{k+1})$,
			\begin{eqnarray}\label{minimizationequality}
				\min_u \left\{ \tilde{\gamma}_k(u) + \frac{1}{2 \lambda}\| u - \tilde{x}_k \|^2 \right\} & = &
				\min_u \left\{ \gamma_k(u) + \frac{1}{2 \lambda}\| u - \tilde{x}_k \|^2 \right\},
			\end{eqnarray}
			and these minimization problems have $y_{k+1}$ as a unique optimal solution; 
			\item[(b)] for every $u  \in \dom h$,
			\[
			\tilde{\gamma}_k(u) -  \phi(u) \le \frac{1}{2}\left(\bar m + \frac{\kappa_0 m}{a_k} \right) \| u - \tilde{x}_k \|^2;
			\]
			\item[(c)]
			$
			x_{k+1} = {\rm{argmin}}_{u\in \Omega} \left\{ a_k \gamma_k(u)  +  \| u - x_k \|^2 /(2\lam) \right\}.
			$
		\end{itemize}
	\end{lemma}
	\begin{proof}
		(a) It clearly follows from \eqref{def2} that ${\gamma}_k(y_{k+1}) = \tilde \gamma_k(y_{k+1})$. By definitions of $\tilde{\gamma}_k$ and ${\gamma}_k$ in (\ref{def1}) and (\ref{def2}) respectively, they are clearly $(\kappa_0 m/a_k)$-strongly convex.  By (\ref{minimize1}) and the definition of $\tilde{\gamma}_k$  in (\ref{def1}), $y_{k+1}$ is the optimal solution to the first minimization problem in (\ref{minimizationequality}).  Since the objective function of this minimization problem is $[(1/\lam) + (\kappa_0 m/a_k)]$-strongly convex, it follows that for all $u \in \mathbb{R}^n$,
		\begin{equation}\label{inequality0}
			\tilde{\gamma}_k(y_{k+1}) + \frac{1}{2 \lambda} \| y_{k+1} - \tilde{x}_k \|^2 + \frac{1}{2}\left( \frac{1}{\lam}+\frac{\kappa_0 m}{a_k}\right)  \| y_{k+1} - u \|^2 \leq  \tilde{\gamma}_k(u) + \frac{1}{2 \lambda}\| u - \tilde{x}_k \|^2.
		\end{equation}
		On the other hand, the definition of $\gamma_k$ in (\ref{def2}) and the relation
		\begin{eqnarray*}
			\| y_{k+1} - \tilde{x}_k \|^2 + \| y_{k+1} - u \|^2 - \| u - \tilde{x}_k \|^2 
			= 2 \langle \tilde{x}_k - y_{k+1}, u - y_{k+1} \rangle.
		\end{eqnarray*}
		imply that
		\begin{eqnarray}\label{equality0} 
			\tilde{\gamma}_k(y_{k+1}) + \frac{1}{2 \lambda} \| y_{k+1} - \tilde{x}_k \|^2 + \frac{1}{2}\left( \frac{1}{\lam}+\frac{\kappa_0 m}{a_k}\right) \| y_{k+1} - u \|^2 =  \gamma_k(u) + \frac{1}{2 \lambda}\| u - \tilde{x}_k \|^2.
		\end{eqnarray}
		Thus, it follows from (\ref{inequality0}) and (\ref{equality0}) that $\gamma_k \leq \tilde{\gamma}_k$.  
		Noting that the objective function in the second minimization problem in \eqref{minimizationequality} is quadratic and using
		the first order optimality condition, we show that $ y_{k+1} $ is a unique optimal solution to the aforementioned problem.
		
		(b) This statement follows from the assumption (A4) and the definition of $\tilde{\gamma}_k(u)$ in (\ref{def1}). 
		
		(c) Using the expressions for $\tilde{x}_k$ and $\hat x_{k+1}$ in (\ref{tildexk}) and (\ref{minimize2prime}), respectively,
		it is easy to see that $\hat x_{k+1}$ is the (unique) global minimizer of the function $ a_k \gamma_k(u) +  \| u - x_k \|^2/(2\lambda)$ over the
		whole space $\R^n$. The definition of $x_{k+1}$ and the previous observation then imply that
		the conclusion of (c) holds.
	\end{proof}
	
	The following result states a recursive inequality that plays an important role in the convergence rate analysis of NC-FISTA.  
	
	\begin{lemma}\label{invariantinequality}
		For every $u \in \Omega$
		and $k \geq 0$, we have
		\[
		\begin{aligned}
			& \lam A_{k+1} \phi(y_{k+1}) + \frac{\kappa_0 m\lam+1}{2}  \| u - x_{k+1} \|^2 + \frac{(1-\lambda  {\cal C}_k)A_{k+1}}{2} \| y_{k+1} - \tilde{x}_k \|^2  \nonumber \\
			& \leq  \lam A_k \gamma_k(y_k) + \lam a_k \gamma_k(u) + \frac{1}{2} \|u - x_k \|^2,
		\end{aligned}
		\]
		where 
		\[
		{\cal C}_k :=  \frac {2 \left[   f(y_{k+1}) - \ell_f(y_{k+1}; \tx_k) \right]}{\|y_{k+1}-\tx_k\|^2}.
		\]
		
	\end{lemma}
	\begin{proof}
		Using the definition of $ {\cal C}_k$, \eqref{def1} and Lemma \ref{gammatildegamma}(a), we conclude that
		\begin{align}
			\lam \phi(y_{k+1}) +& \frac{1-\lam  {\cal C}_k}{2} \| y_{k+1} - \tilde{x}_k \|^2  =  \lam \tilde{\gamma}_k(y_{k+1}) + \left( \frac{1}{2}- \frac{\kappa_0 m\lam}{2a_k}\right)  \| y_{k+1} - \tilde{x}_k \|^2 \nonumber \\
			& \le  \lam {\tilde \gamma}_k(y_{k+1}) + \frac{1}{2}  \| y_{k+1} - \tilde{x}_k \|^2
			=  \lam {\gamma}_k(y_{k+1}) + \frac{1}{2} \| y_{k+1} - \tilde{x}_k \|^2. \label{derivation}
		\end{align}
		On the other hand, using the fact that $\gamma_k$ is convex,
		$y_{k+1}$ is an optimal solution of  (\ref{minimizationequality}),
		and relations \eqref{tildexk} and (\ref{Ak+1ak}), we conclude that for every $u \in \Omega$,
		\begin{align} 
			&  A_{k+1} \left( \lam {\gamma}_k(y_{k+1}) + \frac{1}{2} \| y_{k+1} - \tilde{x}_k \|^2 \right) \nonumber \\
			& \leq  A_{k+1} \left( \lam {\gamma}_k \left(\frac{A_k y_k + a_k x_{k+1}}{A_{k+1}} \right) + \frac{1}{2} \left\| \frac{A_k y_k + a_k x_{k+1}}{A_{k+1}} - \tilde{x}_k \right\|^2 \right) \nonumber \\
			& \leq  \lam A_k \gamma_k(y_k) + \lam a_k \gamma_k(x_{k+1}) +  \frac{A_{k+1}}{2} \left\| \frac{A_k y_k + a_k x_{k+1}}{A_{k+1}} - \tilde{x}_k \right\|^2 \nonumber \\
			& = \lam A_k \gamma_k(y_k) + \lam a_k \gamma_k(x_{k+1}) +  \frac{1}{2} \| x_{k+1} - x_k \|^2 \nonumber \\
			& \leq  \lam {A_k}\gamma_k(y_k) + \lam a_k \gamma_k(u) +  \frac{1}{2} \| u - x_k \|^2 - \frac{\kappa_0 m\lam + 1}{2} \| u - x_{k+1} \|^2,
			\label{derivation2}
		\end{align}
		where the last inequality follows from Lemma \ref{gammatildegamma}(c), the fact that $\gamma_k$ is $(\kappa_0 m/a_k)$-strongly convex
		in view of Lemma \ref{gammatildegamma}(a),
		and hence that
		$\lam a_k \gamma_k(u) +  \| u - x_k \|^2/2$ is $(\kappa_0 m\lam + 1)$-strongly convex.
		The result now follows by combining (\ref{derivation}) and (\ref{derivation2}).
	\end{proof}

	\begin{lemma}\label{consequenceproposition}
		For every $k \geq 1$ and $u \in \dom h$, we have
		\begin{align}
			\sum_{i=0}^{k-1} (1-\lam  {\cal C}_i)A_{i+1} &\| y_{i+1} - \tilde{x}_i \|^2
			\leq  2\lam A_0(\phi(y_0) - \phi(u)) -2\lam A_{k}(\phi(y_{k}) - \phi(u)) \nonumber \\
			& + (\kappa_0 m\lam+1) \left( \|u-x_0\|^2 -  \| u - x_k\|^2 \right)  + \kappa_0 m\lam D^2_h k +
			\bar m \lam D^2_h\sum_{i=0}^{k-1} a_i.\label{inequality3}
		\end{align}
	\end{lemma}
	
	\begin{proof}
		Let $i \geq 0$ and $u \in \dom h$ be given.  
		It follows from Lemma \ref{gammatildegamma}(a)-(b) that we have 
		\begin{align}
			\gamma_i(u)-\phi(u)&\le \tilde \gamma_i(u)-\phi(u)
			\le \frac12\left( \bar m+\frac{\kappa_0 m}{a_i}\right) \|u-\tx_i\|^2. \label{ineq:diff3}
		\end{align}
		Note that for every $ A, a \in \R_+ $ and $ x, y\in \R^n $, we have
		\[
		A\|y\|^2+a\|x\|^2=(A+a)\left\| \frac{Ay+ax}{A+a}\right\|^2 + \frac{Aa}{A+a}\|y-x\|^2.
		\]
		Applying the above identity with $ A=A_i $, $ a=a_i $, $ y=y_i-\tx_i $ and $ x = u-\tx_i $, and using  the definition of $\tilde{x}_i$ in \eqref{tildexk} and the relation \eqref{Ak+1ak},  we obtain
		\begin{align}
			A_i \| y_i - \tilde{x}_i \|^2& + a_i \| u - \tilde{x}_i\|^2 
			=  A_{i+1}\left\|\frac{A_iy_i+a_iu}{A_{i+1}}- \tx_i\right\|^2 + \frac{A_ia_i}{A_{i+1}}\|y_i-u\|^2  \nn \\
			& = \|u-x_i\|^2 + \frac{A_ia_i}{A_{i+1}}\|y_i-u\|^2 
			\le \|u - x_i \|^2 + a_i D^2_h.  \label{ineq:tech}
		\end{align}
		where the inequality follows from the fact that $A_{i+1} = A_i + a_i \ge A_i$ due to \eqref{ak} and the definition of $ D_h $ in \eqref{eq:diam}.
		
		Now,
		using Lemma \ref{invariantinequality}, relations \eqref{ak}, \eqref{ineq:diff3} and \eqref{ineq:tech}, and some simple algebraic manipulations,
		we conclude that for every $i \ge 0$,
		\begin{align*}
			(1 -\lam  {\cal C}_i ) A_{i+1}&\| y_{i+1} - \tilde{x}_i \|^2 + (\kappa_0 m\lam +1) \| u - x_{i+1} \|^2  -  \| u - x_{i} \|^2\\
			&+ 2\lam A_{i+1} (\phi(y_{i+1}) - \phi(u)) -  2\lam A_{i}(\phi(y_{i}) - \phi(u)) \\
			&   \le  2\lam A_i(\gamma_i(y_i) - \phi(y_i)) + 2\lam a_i(\gamma_i(u) - \phi(u)) \nonumber \\
			& \le \lam\left( \bar m + \frac{\kappa_0 m}{a_i}\right) \left( A_i\|y_i-\tx_i\|^2 + a_i\|u-\tx_i\|^2\right) \nonumber \\
			& \le \lam\left( \bar m + \frac{\kappa_0 m}{a_i}\right) \left(\|u-x_i\|^2 + a_i D_h^2\right)  \nonumber \\
			& = \lam\left( \bar m + \frac{\kappa_0 m}{a_i}\right) \|u-x_i\|^2 + (\bar m a_i + \kappa_0 m)\lam D_h^2.
		\end{align*}
		It follows from the above inequality and Lemma \ref{taukproperty} that
		\begin{align*}
			(1 -\lam  {\cal C}_i)  A_{i+1} & \| y_{i+1} - \tilde{x}_i \|^2 
			+ 2\lam A_{i+1} (\phi(y_{i+1}) - \phi(u)) + (\kappa_0 m \lam+1) \| u - x_{i+1} \|^2  \\
			& \le  2\lam A_{i}(\phi(y_{i}) - \phi(u)) 
			+ (\kappa_0 m \lam+1) \|u - x_i \|^2  + (\bar m a_i + \kappa_0 m)\lam D^2_h.
		\end{align*}
		Inequality (\ref{inequality3}) now follows by summing the above inequality from $i=0$ to $i=k-1$ and rearranging terms.
	\end{proof}
	
	The following result develops a convergence rate bound for the quantity $\min_{1 \leq i \leq k} \| v_i \|^2$.
	In view of the stopping criterion in step 3 of NC-FISTA, it plays a crucial role in establishing  an iteration-complexity bound for
	NC-FISTA in Theorem \ref{thm:main}.

	\begin{proposition}\label{lem:convergence}
		Consider the sequences $\{y_k\}$ and $\{v_k\}$ generated by NC-FISTA according to \eqref{minimize1} and \eqref{def:v}, respectively.
		Then, for every $k \ge 1$,
		\begin{equation}\label{incl}
			v_{k} \in \nabla f(y_{k}) + \partial h(y_{k})
		\end{equation}
		and
		\begin{equation}
			\min_{1 \leq i \leq k} \| v_i \|^2
			\le \frac{4(2M+\kappa_0 m)^2}{M-\bar M}
			\left( \frac{ \bar m D^2_h}{k} + \frac{3\kappa_0 m  D^2_h}{k^2} + \frac{3 \left[  2 A_0 (\phi(y_0) - \phi_\ast) + (\kappa_0 m +M) d_0^2\right]}{k^3} \right) \label{eq:convegence rate}
		\end{equation}
		where $M$, $m$, $ \kappa_0 $ and $A_0$ are as described in step 0 of NC-FISTA,
		$D_h$ is defined in \eqref{eq:diam}, $ \bar M $ and $\bar m$ are defined in the paragraph following assumptions (A1)-(A4), and 
		\begin{equation}\label{eq:d0}
			d_0  :=   \inf_{z^\ast \in Z^\ast} \| z^\ast - y_0 \| = \inf_{z^\ast \in Z^\ast} \| z^\ast - x_0 \|.
		\end{equation}	
	\end{proposition}
	
	\begin{proof}
		The first conclusion \eqref{incl} follows from the optimality condition of \eqref{minimize1} and \eqref{def:v}. 
		Next we show the convergence rate bound \eqref{eq:convegence rate} holds.
		First note that $ A_0>0 $ and the relation \eqref{ak} with $k=0$ imply that $a_0 > 1$.
		The assumptions that $\nabla f$ is $\bar M$-Lipschitz continuous (see (A3)), $ M > \bar M$ and
		$\lambda=1/M$ (see step 0 of NC-FISTA), relation (\ref{vk}) and the fact that $\{a_k\}$ is increasing
		then imply that
		\begin{align}\label{inequality11}
			\min_{1 \leq i \leq k} \| v_i \|^2 
			\le \left( \frac{1}{\lambda} + \frac{\kappa_0 m}{a_0}+ \bar M \right)^2 \min_{0 \leq i \leq k-1} \| y_{i+1} - \tilde{x}_i \|^2
			\le (2M+\kappa_0 m)^2 \min_{0 \leq i \leq k-1} \| y_{i+1} - \tilde{x}_i \|^2  .
		\end{align} 
		Moreover, due to the first remark after assumptions (A1)-(A4), there exists $z^* \in Z^*$ such that $\|z^*-x_0\|=d_0$.
		Noting that $z^* \in \dom h$, and using Lemma \ref{consequenceproposition} with $u = z^*$, the fact that $ {\cal C}_k \le \bar M $ for $ k\ge 0 $ and
		$\lambda=1/M$, we conclude that
		\begin{align}
			\frac{M-\bar M}{M}&  \left( \sum_{i=0}^{k-1} A_{i+1}\right) \min_{0 \leq i \leq k-1} \| y_{i+1} - \tilde{x}_i \|^2 
			\le  \sum_{i=0}^{k-1} \left( (1-\lam  {\cal C}_i)A_{i+1} \| y_{i+1} - \tilde{x}_i \|^2\right)  \nonumber \\
			& \le  2\lam A_0 (\phi(y_0) - \phi_\ast)  + \left( \kappa_0 m\lam +1 \right) d_0^2 + \kappa_0 m \lam D^2_h k+ \bar m \lam D^2_h \sum_{i=0}^{k-1} a_i\nonumber \\
			& = \frac1M\left[ 2 A_0 (\phi(y_0) - \phi_\ast)  + \left( \kappa_0 m +M \right) d_0^2 + \kappa_0 m  D^2_h k+ \bar m  D^2_h \sum_{i=0}^{k-1} a_i\right] . \nonumber
		\end{align}
		The bound \eqref{eq:convegence rate} now follows by combining \eqref{inequality11} with  the above inequality and using Lemma A.1 in \cite{jliang2018double}.
	\end{proof}
	
	The following theorem presents the main result of this subsection. It describes an iteration-complexity bound for NC-FISTA involving
	both parameters $M$ and $ m$ as described in its step~0.
	
	\begin{theorem}\label{thm:main}
		Assume that the scalars $M$ and $A_0$ in step 0 of NC-FISTA are such that
		\begin{equation}\label{eq:require}
			\frac{M}{M-\bar M}=\mathcal{O}(1), \quad A_0={\Theta}(1).
		\end{equation}
		Then, NC-FISTA outputs a $\hat \rho$-approximate solution $(\hat y, \hat v)$  in at most
		\begin{equation}\label{eq:complexity}
			\mathcal{O}\left( \left(\frac{M\left(  \phi(y_0)-\phi_*\right)  + M^2d_0^2 }{ \hat{\rho}^2} \right)^{1/3} + \left(\frac{Mm D^2_h}{ \hat{\rho}^2} \right)^{1/2} + \frac{M \bar m D^2_h}{ {\hat{\rho}}^2}  +1 \right)
		\end{equation}
		iterations 
		where 
		$m$ is as in step 0 of NC-FISTA,
		$D_h$ is defined in \eqref{eq:diam}, $\bar m$ is defined in the paragraph following assumptions (A1)-(A4), and
		$d_0$ is defined in \eqref{eq:d0}.
	\end{theorem}
	
	
	\begin{proof}
		Using the assumption that $A_0={\Theta}(1)$ and the definition of $ \kappa_0 $ in \eqref{eq:kappa}, we easily see that
		$ \kappa_0={\Theta}(1) $.
		The iteration-complexity bound in \eqref{eq:complexity} follows immediately from the second result in Proposition \ref{lem:convergence} (see \eqref{eq:convegence rate}), \eqref{eq:require},
		the stopping criterion in step 3 of NC-FISTA, and the facts that $ M\ge m $
		(see step 0 of NC-FISTA) and $ \kappa_0=\Theta (1) $.
	\end{proof}
	
	Note that if a sharper $m \in [\bar m,M]$ is not known and $m$ is simply set to $M$, then \eqref{eq:complexity} reduces to
	\[
	\mathcal{O}\left( \left(\frac{M(\phi(y_0) - \phi_\ast) + M^2d_0^2 }{\hat{\rho}^2} \right)^{1/3} + \frac{M D_h}{\hat{\rho}} +  \frac{M\bar mD^2_h}{{\hat{\rho}}^2} + 1 \right).
	\]
	Clearly, this special case only requires $M$ as the AG method does and achieves the same
	iteration-complexity bound (in regards to the $\Theta(\hat \rho^{-2})$ dominant term).
	%
	
	\section{An adaptive variant of the NC-FISTA}\label{variants}
	
	This section describes the second ACG variant studied in this paper, namely ADAP-NC-FISTA, which, in contrast to NC-FISTA, does not require the knowledge of a curvature pair $(M,m)$
	as input.
	Instead of choosing the parameters $M$ and $ m $ as constants, it generates sequences $\{{\cal C}_k\}$ and $\{m_k\}$
	(see \eqref{req:order}, \eqref{req:lam} and \eqref{req:xi} below).
	
	We begin by describing ADAP-NC-FISTA.
	Note that it requires as input an initial arbitrary pair $(M_0, m_0)$ of positive scalars.

	\noindent\rule[0.5ex]{1\columnwidth}{1pt}
	
	ADAP-NC-FISTA
	
	\noindent\rule[0.5ex]{1\columnwidth}{1pt}
	\begin{itemize}
		\item[0.] Let an initial point $y_0 \in \dom h$, a scalar $ \theta>1 $, a pair $(M_0,m_0)\in\R^2_{++}$ such that $ M_0\ge m_0 $, and a tolerance $\hat \rho>0$ be given, and set $x_0=y_0$, $ A_0=2 $, $ \lam_0=1/M_0 $  and $k=0$;
		\item[1.]  compute	$a_k$ and  $A_{k+1}$ as in \eqref{ak}, $\tx_k$ as in \eqref{tildexk}, 
		\begin{equation}\label{def:ty}
			\ty_k = \frac{A_k y_k + a_k y_0}{A_{k+1}},
		\end{equation}
		and
		\begin{equation}\label{def:L}
			\underline m_{k+1} = \max\left\lbrace \frac{2 [\ell_f(\ty_k; \tx_k)-f(\ty_k) ]} {\|\ty_k-\tx_k\|^2}, 0\right\rbrace;
		\end{equation}
		\item[2.] 
		call the subroutine SUB$ (\theta,\lam_k,m_k) $ stated below to compute $(\lam_{k+1}, m_{k+1}) = (\lam,m)$ satisfying
		\begin{align}
			& \lam\le \lam_k, \quad m \ge m_k, \label{req:order}\\
			& \lam C_k(\lambda,m)\le 0.9, \label{req:lam}\\
			& 2m\left( \lam_k - \frac{\lam}{a_k}\right) \ge \underline m_{k+1}\lam, \label{req:xi}
		\end{align}
		where
		\begin{align}
			C_k(\lambda,m)&:=\frac{2[f(y_k(\lam,m))-\ell_f(y_k(\lambda,m);\tx_k)]}{\|y_k(\lambda,m)-\tx_k\|^2}, \label{def:Mk}\\
			y_k(\lambda,m)&:=\argmin_u \left\{ \ell_f(u;\tx_k)+h(u)+\frac12\left( \frac1{\lam}+\frac{2m}{a_k}\right)\|u-\tx_k\|^2 \right\}, \label{def:yk}
		\end{align}
		and go to step 3;
		\item[3.] compute 
		\begin{align}
			& y_{k+1}=y_k(\lam_{k+1},m_{k+1}), \quad  {\cal C}_{k+1}= C_k(\lam_{k+1},m_{k+1}),  \label{eq:y,M}\\
			& x_{k+1}= P_\Omega\left(\frac{(a_k+2m_{k+1}\lam_{k+1})y_{k+1}-(a_k-1)y_k}{2m_{k+1}\lam_{k+1} + 1}  \right), \nonumber \\
			& v_{k+1} =  \left( \frac{1}{\lam_{k+1}}+\frac{2m_{k+1}}{a_k}\right)  (\tilde{x}_k - y_{k+1}) + \nabla f(y_{k+1}) - \nabla f(\tilde{x}_k); \label{eq:v}
		\end{align}
		if $\|v_{k+1} \| \le \hat \rho$ then output  $(\hat y,\hat v)=(y_{k+1},v_{k+1})$ and {\bf stop}; otherwise, set $k \leftarrow k+1$ and go to step~1.
	\end{itemize}
	\rule[0.5ex]{1\columnwidth}{1pt}
	
	\vspace{5mm}
	We will  now describe the subroutine SUB$ (\theta,\lam,m) $ used in step 2 of ADAP-NC-FISTA to compute $(\lam, m)$ satisfying conditions  \eqref{req:order}-\eqref{req:xi}.
	\vspace{5mm}
	
	\noindent\rule[0.5ex]{1\columnwidth}{1pt}
	
	SUB$ (\theta,\lam,m) $
	
	\noindent\rule[0.5ex]{1\columnwidth}{1pt}
	
	\begin{itemize}
		\item[0.] Compute $ C_k(\lam,m) $ and $ y_k(\lam,m)$ according to \eqref{def:Mk} and \eqref{def:yk}, respectively;
		\item[1.] \textbf{if} $ (\lam,m) $ satisfy
		both \eqref{req:lam} and \eqref{req:xi}, then output $(\lam,m) $ and {\bf stop};
		\textbf{otherwise}, if \eqref{req:lam} is not satisfied then set
		\begin{align}
			\lambda^+ \leftarrow  \min \left\{ \frac{\lam}{\theta}, \frac{0.9}{ C_k(\lam,m)} \right \}; \label{eq:lamupdate}
		\end{align}
		if \eqref{req:xi} is not satisfied then set 
		\begin{align}
			m ^+\leftarrow 2 m; \label{eq:xiupdate}
		\end{align}
		\item[2.]  set $(\lam,m)=(\lam^+,m^+)$ and go to step 0.
	\end{itemize}
	\rule[0.5ex]{1\columnwidth}{1pt}
	
	We now make a few remarks about ADAP-NC-FISTA.
	First, ADAP-NC-FISTA consists of two types of iterations, namely, the ones indexed by $k$ that we refer to as outer iterations
	and the ones performed inside SUB$ (\theta,\lam,m) $ that we refer to as inner iterations. 
	Second, each inner iteration performs exactly
	one resolvent evaluation of $h$ to compute $y_k(\lam,m)$.
	Third, when the update \eqref{eq:lamupdate} is performed,
	the quantity $C_k(\lam,m)$ in the right hand side of \eqref{eq:lamupdate} is always positive due to
	the fact that \eqref{req:lam} is not satisfied  and,
	as a consequence, $\lam^+$ is well-defined and positive.
	Fourth, the choice of $A_0=2$, \eqref{ak} with $k=0$ and the fact that $\{a_k\}$ is increasing imply that $a_k \ge a_0=2$.
	Fifth, if $f$ is convex, and hence $\bar m = 0$, and $m_0$ is set to $0$ in ADAP-NC-FISTA, then it can be  easily seen that
	the adaptive search for $\lambda_k$ is equivalent to the adaptive search for the quantity $L_k$ in \cite{beck2009fast} via the correspondence $L_k=1/\lam_k$.
	Thus, ADAP-NC-FISTA reduces to FISTA with backtracking when $\bar m=0$.
	

	
	The following lemma states some properties of ADAP-NC-FISTA.
	
	\begin{lemma}\label{observation}
		The following statements hold for ADAP-NC-FISTA:
		\begin{itemize}
			\item[(a)]
			for every $k \geq 0$ and $\lam, m>0$, the quantities $C_k(\lam,m)$ and $ {\cal C}_{k+1}$
			defined in (\ref{def:Mk}) and (\ref{eq:y,M}), respectively, lie in $[-\bar m, \bar M]$;
			\item[(b)] for every $k \geq 0$, the quantity $\underline m_{k+1}$
			defined in (\ref{def:L}) lies in $[0, \bar m]$;
			
			\item[(c)] for every $k \ge 1$,
			\[
			{\cal C}_{k}\lambda_{k} \le 0.9, \quad 2m_{k}\lam_{k-1}\ge \underline m_{k}\lam_{k}+\frac{2m_{k}\lam_{k}}{a_{k-1}};
			\]
			\item[(d)]
			$\{\lambda_k\} $ is non-increasing and $\{m_k\} $ is non-decreasing;
			\item[(e)] for every $k \ge 0$,
			\begin{equation} \label{eq:barlam-barxi}
				\lam_k\ge \underline \lam :=\min\left\{ \frac{0.9}{\theta\bar M}, \lam_0 \right \}, \qquad
				m_k\le \max\{2\bar m, m_0\};
			\end{equation}
		\end{itemize}
	\end{lemma}
	
	\begin{proof}
		(a)-(b) It follows from \eqref{ineq:Lips} (resp., \eqref{ineq:curv}) and
		the fact that $\bar M$ (resp., $ \bar m $) is the smallest scalar $ M $ (resp., $ m $) satisfying \eqref{ineq:Lips} (resp., \eqref{ineq:curv})
		that $ C_k(\lam,m) $ and $ {\cal C}_{k+1}$ (resp., $ \underline m_{k+1} $) is bounded above by $\bar M$ (resp., $ \bar m $). The quantities $ C_k(\lam,m) $ and $ {\cal C}_{k+1}$ are bounded below by $ -\bar m $ follows from $\bar m$ satisfying (\ref{ineq:curv}), and $ \underline m_{k+1}$ is non-negative due to \eqref{def:L}.
		
		(c) The two conclusions follow from requirements \eqref{req:lam} and \eqref{req:xi}. 
		
		(d) The requirements in \eqref{req:order} on $ (\lam,m) $ immediately imply the two conclusions.
		
		(e) We first prove the first inequality in (\ref{eq:barlam-barxi}). Indeed, assume for contradiction that it does not hold and
		let $ \hat k $ be the smallest $k \ge 0$ such that $ \lam_k<\underline \lam $.
		Since
		$\underline \lam  \le \lam_0$ in view of the definition of $\underline \lam$ in \eqref{eq:barlam-barxi},
		it follows from the definition of $\hat k$ that
		$\lam_{\hat k}$ is obtained from \eqref{eq:lamupdate}, i.e.,
		\begin{eqnarray}\label{ineq:lambdakhat2}
			\lam_{\hat k}=\lam^+:=\min\left\lbrace \frac{\lam}{\theta},\frac{0.9}{C_{\hat k-1}(\lam,m)}\right\rbrace 
		\end{eqnarray}
		for some $ (\lam, m) \in (0,\lam_0] \times \R_{++}$ such that \eqref{req:lam} does not hold for the pair $(\lambda, m)$ where $k = \hat{k} - 1$ in (\ref{req:lam}).   Hence $ C_{\hat k-1}(\lam,m) > 0$  in view of the third remark following SUB$ (\theta,\lam,m) $. 
		Moreover, it follows from the definition of $\underline \lam$ in \eqref{eq:barlam-barxi},  statement (a)
		and the facts that $\theta>1$ and $ C_{\hat k-1}(\lam,m) > 0$ that
		\begin{eqnarray}\label{ineq:lambdakhat}
			\lam_{\hat k} < \underline{\lambda} \leq \frac{0.9}{\theta\bar M} < \frac{0.9}{\bar M} \le \frac{0.9}{C_{\hat k-1}(\lam,m)}.
		\end{eqnarray}
		Clearly, \eqref{ineq:lambdakhat2}
		and \eqref{ineq:lambdakhat} imply that $\lam_{\hat k}= \lam/\theta$. On the other hand,
		the fact that $\lam$ does not satisfy \eqref{req:lam} and statement (a) imply that
		$ \lam > 0.9/C_{\hat k-1}(\lam,m) \ge 0.9/\bar M $
		and hence that
		$ \lam_{\hat k} = \lam/\theta > 0.9/\theta\bar M \ge \underline \lam $
		due to the definition of $\underline \lam$. Since the latter inequality contradicts
		our initial assumption, the first inequality in (\ref{eq:barlam-barxi}) follows.
		To prove the second inequality in (\ref{eq:barlam-barxi}), assume for contradiction that it does not hold and let ${\bar{k}} \geq 0$ be such that $m_{{\bar{k}}} > \max\{2\bar m, m_0\} $.  It follows that $m_{\bar{k}} >  m_0$ by the definition of $\hat{m}$ in (\ref{eq:barlam-barxi}), which, in view of \eqref{eq:xiupdate}, implies that ${\bar{k}} \geq 1$ and $m_{{\bar{k}}} = 2 m$ for some $ m\in \R_{++} $ that does not satisfied \eqref{req:xi}, i.e., $m$ satisfies
		\begin{eqnarray}\label{ineq:xihathatk}
			2m \lambda_{{{\bar{k}}}-1} < \underline m_{{\bar{k}}} \lambda + \frac{2m\lam}{a_{{{\bar{k}}}-1}}.
		\end{eqnarray}   
		It then follows from (\ref{ineq:xihathatk}), $\underline{m}_{\bar{k}} \leq \bar{m}$ due to  statement (b), $\lambda \leq \lambda_{\bar{k}-1}$, and $a_{{\bar{k}}-1} \geq a_0 = 2$ that $m <  \bar{m}$.  The latter inequality and the fact that $m_{{\bar{k}}} = 2 m$  imply that $m_{\bar{k}} < \max\{2\bar m, m_0\}$, which contradicts our initial assumption.  Hence the second inequality in (\ref{eq:barlam-barxi}) follows.  
	\end{proof}
	%

	
	We have the following technical results that lead to Proposition \ref{lem:xi}, which then allows us to establish the iteration-complexity result for ADAP-NC-FISTA in Theorem \ref{thm:ADAP-NC-FISTA}.
	
	\begin{lemma}\label{lem:technical}
		For every $k \geq 0$ and $u \in \R^n$, we define
		\begin{equation}\label{def:gamma}
			\gamma_k(u) := \tilde\gamma_k(y_{k+1}) +  \frac{1}{\lam_{k+1}}\langle \tilde{x}_k - y_{k+1}, u - y_{k+1} \rangle + \frac{m_{k+1}}{a_k} \| u - y_{k+1} \|^2
		\end{equation}
		and
		\begin{equation}\label{def:tgamma}
			\tilde \gamma_k(u) :=  \ell_f(u;\tilde{x}_k) + h(u) + \frac{m_{k+1}}{a_k} \| u - \tilde{x}_k \|^2.
		\end{equation}
		Then, for every $k \ge 0$, we have:
		\beq 
		\displaystyle A_k\gamma_k(y_k) + a_k\gamma_k(y_0) \le
		A_{k+1} \gamma_k\left( \ty_k \right) + m_{k+1}\|y_k - y_0\|^2, \label{ineq:technical1}
		\eeq
		\beq
		\displaystyle A_{k+1}\phi\left( \ty_k \right)  - A_{k} \phi(y_k) -a_k \phi(y_0) \le \frac{\bar m a_k}{2}   \left\| y_k-y_0 \right\|^2, \label{ineq:technical2}
		\eeq
		\beq 
		\displaystyle \gamma_k(\ty_k)-\phi(\ty_k) \le \frac{m_{k+1}\lam_k}{A_{k+1}\lam_{k+1}}\|y_0-x_k\|^2. \label{ineq:technical3}
		\eeq
	\end{lemma}
	\begin{proof}
		Note that for any quadratic function $ \gamma:\R^n \to \R $ with a quadratic term $ \alpha\|\cdot\|^2 $, every $ A, a \in \R_+ $ and $ x, y\in \R^n $, we have
		\[
		A\gamma(y)+a\gamma(x)=(A+a)\gamma\left(  \frac{Ay+ax}{A+a}\right) + \frac{Aa}{A+a}\alpha\|y-x\|^2.
		\]
		Applying the above identity with $ \gamma=\gamma_k $, $ A=A_k $, $ a=a_k $, $ y=y_k $ and $ x = y_0 $, and using  the definition of $\tilde{y}_k$ in \eqref{def:ty} and the relation \eqref{Ak+1ak},  we obtain
		\[
		A_k\gamma_k(y_k) + a_k\gamma_k(y_0) = A_{k+1} \gamma_k\left( \ty_k \right) + \frac{m_{k+1}A_k}{A_{k+1}}\|y_k - y_0\|^2\le A_{k+1} \gamma_k\left( \ty_k \right) + m_{k+1}\|y_k - y_0\|^2
		\]
		where the inequality follows from the fact that $ A_k\le A_{k+1} $.
		Inequality (\ref{ineq:technical1}) then follows.  We now show (\ref{ineq:technical2}). Due to the convexity of $h$, and relations (\ref{ak}) and (\ref{def:ty}), we have
		\[
		A_{k+1} h(\ty_k) - A_k h(y_k) - a_k h(y_0) \leq 0.
		\]
		It follows from $ \phi=f+h $, the above inequality, the fact that $A_k \leq A_{k+1}$, and relations
		\eqref{ineq:curv}, (\ref{ak}) and \eqref{def:ty} that
		\begin{align*}
			A_{k+1}\phi\left( \ty_k \right)  - A_{k} \phi(y_k) -a_k \phi(y_0)&\le A_{k+1}f\left(\ty_k \right)  - A_{k} f(y_k) -a_k f(y_0)  \\
			& \le \frac{\bar m A_k a_k}{2A_{k+1}}   \left\| y_k-y_0 \right\|^2  \le \frac{\bar m a_k}{2}   \left\| y_k-y_0 \right\|^2.
		\end{align*}
		Next, we show (\ref{ineq:technical3}). Using similar arguments as in the proof of Lemma \ref{gammatildegamma}(a), we have $ \gamma_k(u)\le \tilde \gamma_k(u) $ for every $ u\in \dom h $.  Hence, using \eqref{def:tgamma}, \eqref{def:L}, \eqref{def:ty}, \eqref{tildexk} and Lemma \ref{observation}(c) that for every $ k \ge 0 $, we have 
		\begin{align*}
			&\gamma_k(\ty_k)-\phi(\ty_k)\le \tilde \gamma_k(\ty_k)-\phi(\ty_k)
			=\ell_f(\ty_k;\tx_k)-f(\ty_k)+\frac{m_{k+1}}{a_k}\|\ty_k-\tx_k\|^2 \\
			&\le \frac12\left( \underline m_{k+1}+\frac{2m_{k+1}}{a_k}\right) \|\ty_k-\tx_k\|^2 = \frac{1}{2A_{k+1}}\left( \underline m_{k+1}+\frac{2m_{k+1}}{a_k}\right) \|y_0-x_k\|^2 
			\le \frac{m_{k+1}\lam_k}{A_{k+1}\lam_{k+1}}\|y_0-x_k\|^2.
		\end{align*}
		Inequality (\ref{ineq:technical3}) then follows.
	\end{proof}
	
	\begin{proposition}\label{lem:xi}
		For every $ k\ge 1 $, we have
		\begin{equation}\label{ineq:sum2}
			\frac{1}{20}\left( \sum_{i=0}^{k-1}\frac{A_{i+1}}{m_{i+1}}\right) \min_{0 \leq i \leq k-1} \| y_{i+1} - \tilde{x}_i \|^2 \le 
			\lam_0D_h^2 \left( k+\bar m\sum_{i=0}^{k-1}\frac{a_i}{2m_{i+1}}\right) + \frac{2\lam_0}{m_0}A_k(\phi(y_0)-\phi_*).
		\end{equation}
	\end{proposition}
	\begin{proof}
		Using similar arguments as in the proof of Lemma \ref{invariantinequality} and the definition of $ {\cal C}_i $ in \eqref{eq:y,M}, we conclude that  for every $i \geq 0$ and $u \in \Omega$,
		\begin{align}
			& 2\lam_{i+1} A_{i+1} \phi(y_{i+1}) + \left(  2m_{i+1}\lam_{i+1} + 1 \right) \| u - x_{i+1} \|^2 + (1-\lam_{i+1}  {\cal C}_{i+1}) A_{i+1} \| y_{i+1} - \tilde{x}_i \|^2 \nn  \\
			& \leq  2\lam_{i+1} A_i \gamma_i(y_i) + 2\lam_{i+1}  a_i \gamma_i(u) + \|u - x_i \|^2, \label{ineq:recursive}
		\end{align}
		where $\gamma_i$ and $\tilde \gamma_i$ are defined by \eqref{def:gamma} and \eqref{def:tgamma}, respectively.
		Using the relation \eqref{ineq:recursive} with $u=x_0$, Lemmas \ref{lem:technical}, \ref{observation}(c)-(d),
		the facts that $ x_0=y_0 $ and $ \lam_i\le \lam_0 $ for $ i\ge 0 $, and the definition of $ D_h $ in \eqref{eq:diam} we conclude that for every $0 \le i \le k-1$,
		\begin{align*}
			\frac{1}{10} A_{i+1} &\| y_{i+1} - \tilde{x}_i \|^2  + 
			\left[ 2\lam_{i+1}A_{i+1} (\phi(y_{i+1}) - \phi(y_0)) + (2m_{i+1}\lam_{i+1} +1)\|x_0 - x_{i+1} \|^2 \right] \nonumber \\
			& \ \ -  \left[ 2\lam_{i+1} A_{i}(\phi(y_{i}) - \phi( y_0)) + \|  x_0 - x_{i} \|^2 \right] \nonumber  \\
			&   \le  2\lam_{i+1} A_i(\gamma_i(y_i) - \phi(y_i)) + 2\lam_{i+1} a_i(\gamma_i(y_0) - \phi(y_0) ) \nonumber \\
			& = 2\lam_{i+1} [ A_i \gamma_i(y_i) + a_i \gamma_i(y_0) - A_{i+1}\phi(\ty_i) ] + 2\lam_{i+1} [A_{i+1}\phi(\ty_i) - A_i \phi(y_i) - a_i \phi(y_0) ] \nn \\
			&  \le   2\lam_{i+1}\left[ A_{i+1}( \gamma_i\left(\ty_i \right)  -\phi\left(\ty_i \right)) + m_{i+1}\|y_i - y_0\|^2 \right] 
			+ \bar m a_i\lam_{i+1}   \left\| y_i-y_0 \right\|^2   \nonumber \\
			&  \le   2m_{i+1}\lam_i \left\| y_0 - x_i\right\|^2 +2\lam_{i+1} m_{i+1}\|y_i - y_0\|^2
			+  \bar m a_i\lam_{i+1}   \left\| y_i-y_0 \right\|^2  \nonumber \\
			& \le  2m_{i+1}\lam_i \|x_0-x_i\|^2  +   (2m_{i+1}+\bar m  a_i ) \lam_0 D_h^2 \label{ineq:key1} 
		\end{align*}
		where the second inequality follows from \eqref{ineq:technical1} and \eqref{ineq:technical2}, the third inequality follows from \eqref{ineq:technical3}.
		Dividing the above inequality by $ 2m_{i+1} $, rearranging terms and using the fact that, by Lemma \ref{observation}(d), $ m_i\le m_{i+1}$,
		we obtain
		\begin{align*}
			\frac{A_{i+1}}{20m_{i+1}} \| y_{i+1} - \tilde{x}_i \|^2 \le & 
			\left[ \frac{\lam_{i}}{m_{i}} A_{i}(\phi(y_{i}) - \phi( y_0)) + \left ( \frac{1}{2m_{i}}+\lam_i \right)  \|  x_0 - x_{i} \|^2 \right] \nonumber \\
			& -\left[ \frac{\lam_{i+1}}{m_{i+1}}A_{i+1} (\phi(y_{i+1}) - \phi(y_0)) + \left( \frac{1}{2m_{i+1}}+\lam_{i+1} \right)  \|x_0 - x_{i+1} \|^2 \right] \nonumber \\
			&+\left( \frac{\lam_{i}}{m_{i}}-\frac{\lam_{i+1}}{m_{i+1}}\right) A_i(\phi(y_0)-\phi(y_i))+\left( 1+ \frac{\bar m a_i}{2m_{i+1}}\right) \lam_0D_h^2. 
		\end{align*}
		Summing the above inequality from $i=0$ to $i=k-1$ and using the facts $\phi(y_i) \ge \phi_*$ for $ i\ge 0 $ and
		$\{\lam_i/m_i\}$ is non-increasing due to Lemma \ref{observation}(d), we obtain
		\begin{align*}
			\frac{1}{20} &\left( \sum_{i=0}^{k-1}\frac{A_{i+1}}{m_{i+1}}\right)  \min_{0 \leq i \leq k-1} \| y_{i+1} - \tilde{x}_i \|^2 \le 
			\frac{\lam_k}{m_k}A_k(\phi(y_0)-\phi(y_k)) -  \left( \frac{1}{2m_{k}} + \lam_k \right) \|x_0 - x_{k} \|^2
			\nonumber \\
			&\ \ \ + \sum_{i=0}^{k-1} \left( \frac{\lam_i}{m_i}-\frac{\lam_{i+1}}{m_{i+1}}\right)  A_i (\phi(y_0)-\phi(y_i)) 
			+ \lam_0D_h^2 \left( k+\bar m\sum_{i=0}^{k-1}\frac{a_i}{2m_{i+1}}\right)  \nonumber  \\
			&\le  \frac{\lam_0}{m_0}A_k(\phi(y_0)-\phi_*) + (\phi(y_0)-\phi_*)  \sum_{i=0}^{k-1} \left( \frac{\lam_i}{m_i}-\frac{\lam_{i+1}}{m_{i+1}}\right)  A_i 
			+ \lam_0D_h^2 \left( k+\bar m\sum_{i=0}^{k-1}\frac{a_i}{2m_{i+1}}\right). 
		\end{align*}
		Now, using the fact that $ \{A_k\} $ is increasing and $\{\lam_k/m_k\}$ is non-increasing, we have
		\[
		\sum_{i=0}^{k-1} \left( \frac{\lam_i}{m_i}-\frac{\lam_{i+1}}{m_{i+1}}\right)  A_i  \le \frac{\lam_0}{m_0}A_0 + \sum_{i=1}^{k-1} (A_i-A_{i-1}) \frac{\lam_i}{m_i} \le \frac{\lam_0}{m_0}A_0 + \sum_{i=1}^{k-1} (A_i-A_{i-1}) \frac{\lam_0}{m_0}  \le \frac{\lam_0}{m_0}A_k.
		\]
		Combining the above two inequalities, we then conclude that \eqref{ineq:sum2} holds.
	\end{proof}
	
	The next theorem is the main result of this section presenting the iteration-complexity for finding a $ \hat \rho $-approximate solution of \eqref{eq:probintro} by ADAP-NC-FISTA.

	\begin{theorem}\label{thm:ADAP-NC-FISTA} 
		The following statements hold:
		\begin{itemize}
			\item[(a)]
			every iterate $(y_k,v_k)$ generated by ADAP-NC-FISTA satisfies
			\begin{eqnarray*}
				v_{k} \in \nabla f(y_{k}) + \partial h(y_{k});
			\end{eqnarray*}
			moreover, ADAP-NC-FISTA outputs a  $\hat \rho$-approximate solution $(\hat y,\hat v)$ in a number of
			outer iterations ${\cal T}$  bounded by
			\begin{equation}\label{cmplx}
				{\cal T} = \mathcal{O}\left( 
				\left(\frac{ C_1 \bar M [ \phi(y_0)-\phi_*] }{\hat{\rho}^2} \right)^{1/3} +
				\left( \frac{C_1 \bar M m_0 D_h^2}{\hat \rho^2} \right)^{1/2} +
				\frac{C_1 \bar M \left[ \bar mD_h^2+\phi(y_0)-\phi_*\right] }{\hat{\rho}^2}+1 \right)
			\end{equation}
			where $D_h$ is defined in \eqref{eq:diam}, $\bar m$ and $\bar M$ are defined in the paragraph following assumptions (A1)-(A4), and
			\begin{equation}\label{def:C1}
				C_1:= C_2
				\max\left\{ \frac{\bar m}{m_0},1 \right\}, \quad C_2:= \left[ \sqrt{\frac{M_0}{\bar M}} +\sqrt{\frac{\bar M}{M_0}} \, \right]^2;
			\end{equation}
			\item[(b)]
			if $ m_0\ge \bar m $,  then an alternative bound on ${\cal T}$ is
			\begin{equation}\label{cmplx2}
				{\cal T} = \mathcal{O}\left( \left(\frac{ C_2 \bar M \left[ \phi(y_0)-\phi_* + M_0 d_0^2\right]   }{\hat{\rho}^2} \right)^{1/3} +\left( \frac{C_2 \bar M m_0 D_h^2}{\hat \rho^2}\right)^{1/2} + \frac{C_2 \bar M \bar m D_h^2 }{\hat{\rho}^2}+1 \right),
			\end{equation}
			where $ C_2 $ is defined in \eqref{def:C1} and $ d_0 $ is defined in \eqref{eq:d0};
			\item[(c)] the total number of inner iterations, and hence resolvent evaluations of $h$, performed by ADAP-NC-FISTA is bounded by
			\begin{align}
				{\cal T} + \mathcal{O}\left( 
				\log_1^+ \left( \max \left\{ \frac{\bar M}{M_0} , \frac{\bar m}{m_0} \right \}  \right) \right) \label{resolvent}
			\end{align}
			where $ \log^+_1(\cdot) $ is defined in Subsection \ref{sec:DefNot}.
		\end{itemize}
	\end{theorem}
	
	\begin{proof}
		(a) The first conclusion follows from the same argument as in the proof of Proposition \ref{lem:convergence}. Using the facts that $ a_k\ge a_0 = 2$ from the fourth remark after SUB$ (\theta,\lam,m) $ and Lemma \ref{observation}(e), we have
		\[
		\frac{1}{\lam_{k+1}}+\frac{2m_{k+1}}{a_k}
		\le \frac{1}{\underline \lam}+\max\{2\bar m, m_0\}
		\]
		for every $k \ge 0$. This conclusion together with the definition of $ \bar M $ in the paragraph following assumptions (A1)-(A4), assumption (A3) and  \eqref{eq:v} then implies that
		\begin{align} \nonumber
			\min_{1 \leq i \leq k}\|v_i\| & \le\min_{0 \leq i \leq k-1} \left( \frac{1}{\lam_{i+1}}+\frac{2m_{i+1}}{a_i}+\bar M\right) \|y_{i+1}-\tx_i\| \\
			&\le \left( \frac{1}{\underline \lam}+\max\{2\bar m, m_0\}+\bar M\right) \min_{0 \leq i \leq k-1}\|y_{i+1}-\tx_i\|. \label{eq:obss}
		\end{align}
		Moreover, using the definition of $ \underline \lam $ in \eqref{eq:barlam-barxi}, the facts that $ \bar m \le \bar M $ and $ \lam_0=1/M_0 $, and the definition of $ C_1 $ in \eqref{def:C1}, we have
		\[
		\frac{1}{\underline \lam}+\max\{2\bar m, m_0\}+\bar M\le \left( \frac{\theta}{0.9}+3\right) \left( M_0+\bar M\right)  \le (2\theta+5)\sqrt{\frac{C_1\bar M M_0 m_0}{\max\{2\bar m, m_0\}}}.
		\]
		Using Proposition \ref{lem:xi}, Lemma \ref{observation} (d)-(e), the above two inequalities, the fact that $ A_k=A_0 + \sum_{i=0}^{k-1}a_i $ due to \eqref{ak}, and rearranging terms, we obtain
		\begin{align*}
			& \frac{1}{20}\left(\sum_{i=0}^{k-1} A_{i+1}\right) \min_{1 \leq i \leq k} \| v_i \|^2 \\
			& \le (2\theta+5)^2C_1 \bar M \left[ 2A_0(\phi(y_0)-\phi_*)+m_0 D_h^2k +\left[ \frac{\bar mD_h^2}{2}+2(\phi(y_0)-\phi_*)\right]  \sum_{i=0}^{k-1}a_i\right] . 
		\end{align*}
		The complexity bound  \eqref{cmplx} now follows immediately from the above inequality and Lemma A.1 in \cite{jliang2018double}.
		
		(b)
		The proof of this statement is similar to  the proof of  (a) except that  Proposition \ref{lem:xi2} is used in place of Proposition \ref{lem:xi}.
		
		(c)
		It suffices to argue that the total number of times that the pair $(\lam,m)$ is updated inside
		all calls to the subroutine SUB$ (\theta,\lam,m) $ is bounded by the second term in \eqref{resolvent}.
		Indeed, this assertion follows from the following facts: 
		the initial value of $(\lam,m)$ is $(\lam_0,m_0)$ (see step 0 of ADAP-NC-FISTA);
		in view of \eqref{req:lam} and \eqref{req:xi}, the pair $ (\lam,m) $ is no longer updated whenever $ \lam \le 0.9/\bar M$ and $m \geq 2\bar{m}$, and;
		due to \eqref{eq:lamupdate} and \eqref{eq:xiupdate}, $\lambda$ is reduced by a factor less than or equal to $\theta>1$ and $m$ is increased by a factor of 2 each time either one of them is updated.
		%
	\end{proof}

	We now make two remarks about  ADAP-NC-FISTA in light of NC-FISTA.
	First, in contrast to NC-FISTA, the input pair $ (M_0,m_0) $ of ADAP-NC-FISTA can
	be an arbitrary pair  in $\R^2_{++}$.
	Second, if $ (M,m) $ denotes a pair as in step 0 of NC-FISTA, then it can be easily seen that
	$ (M_0,m_0) = (M,m)$ satisfies the assumption of Theorem \ref{thm:ADAP-NC-FISTA}(b) and
	the complexity bound \eqref{cmplx2} for ADAP-NC-FISTA with input pair $ (M_0,m_0) = (M,m)$ reduces to
	the complexity bound \eqref{eq:complexity} for NC-FISTA.

	We end this section by making a few final remarks about the iteration-complexity bound derived in Theorem \ref{thm:ADAP-NC-FISTA}(b)
	for the case in which $ M_0 = {\cal O}(\bar M)$. First, in this case,
	the dominant term of the complexity bound \eqref{cmplx2} is 
	$ {\cal O}\left( \bar M^2\bar m D_h^2/(M_0\hat \rho^2)\right)$, and hence
	it increases as $ M_0 $ decreases.
	Second, the best choice of $ M_0 $ that minimizes the constant $C_2$ in \eqref{def:C1} is $ M_0=\Theta(\bar M) $.
	However, computational experiments indicate that taking smaller values for $M_0$ improves the performance of the method.
	One reason that may explain this phenomenon is that the constant $\bar M$ that appears in \eqref{eq:obss},
	and as a consequence in  $C_1$, $C_2$, and the other terms that appear in the bounds \eqref{cmplx} and \eqref{cmplx2},
	is very conservative and close examination of the proof of Theorem \ref{thm:ADAP-NC-FISTA} shows that it
	can actually be replaced by the sharper (and potentially smaller) quantity
	\[
	L_k := \frac {\|\nabla f(y_{\hat k})-\nabla f(\tx_{\hat k-1})\| }
	{ \|y_{\hat k}-\tx_{\hat k-1}\| }, 
	\]
	where $\hat k = \argmin_i \{ \|y_{i}-\tx_{i-1}\| : 1 \le  i \le k\}$.

	\section{Computational results} \label{sec:computResults}
	This section reports experimental results obtained by our implementation of NC-FISTA, ADAP-NC-FISTA, and three variants
	of the latter method, on four problems that are instances of the SNCO problem \eqref{eq:probintro}, namely: nonconvex quadratic programming problem in both vector (Subsection \ref{subsec:vector}) and matrix versions (Subsection \ref{subsec:matrix}), matrix completion (Subsection \ref{subsec:MC}) and nonnegative matrix factorization (NMF, Subsection \ref{subsec:NMF}). Note that NMF is a problem for which $ \dom h $ is unbounded.
	
	We start by describing the three variants of ADAP-NC-FISTA
	considered in our computational benchmark, namely, R-ADAP-NC-FISTA, ADAP-NC-FISTA-BB and R-ADAP-NC-FISTA-BB.
		The first one is a restart variant of ADAP-NC-FISTA, namely,
		it restarts the latter method with input $y_0=y_k$ and $(M_0,m_0)=(M_0,m_k) $ whenever $\phi(y_{k+1}) \ge \phi(y_k)$ (hence, without resetting $k$ to $0$,
		this is equivalent to rejecting $y_{k+1}$ and 
		setting $ x_k=y_k $, $ A_k=A_0 $ and  $\lam_k=\lam_0$).
		The last two variants are heuristic variants of ADAP-NC-FISTA and R-ADAP-NC-FISTA,
		respectively, which invokes in step 2 the subroutine SUB with input
		$ (\theta,\tilde \lam_k,m_k) $ where 
		\[
	    \tilde \lam_k = \left\{ \begin{array}{cc}  
	    \lam_k^{BB}:=\frac{\inner{s_{k-1}}{g_{k-1}}}{\|g_{k-1}\|^2}, & \mbox{if $\lam_k^{BB}>0$};  \\ [.1in]
        \frac{1}{M_0}, & \mbox{otherwise} 
        \end{array} \right.
	    \]
		where $ s_{k-1}=\tx_{k-1}-y_k $ and $ g_{k-1}= \nabla f(\tx_{k-1})- \nabla f(y_k) $.
		
	
	For the sake of simplicity, we use the abbreviations NC, AD, AD(B), RA and RA(B) to refer to NC-FISTA, ADAP-NC-FISTA, ADAP-NC-FISTA-BB, R-ADAP-NC-FISTA and R-ADAP-NC-FISTA-BB, respectively, both in the discussions	and tables below.
	The triples $ (M, m, A_0) $ and $(M_0, m_0, \theta)$ which are used as input for NC and AD, respectively, depend on the problem under consideration and are described
	in the four subsections below. Moreover,
	AD(B), RA and RA(B) use the same input triple as AD.
	
	
	We compare our methods with four others: the AG method proposed in \cite{nonconv_lan16}, the NM-APG method proposed in \cite{Li_Lin2015}, and the UPFAG and UPFAG-BB methods proposed in \cite{LanUniformly}.
	Note that
	all four methods are natural extensions of ACG variants for solving convex programs to the context of nonconvex optimization problems.
	For the sake of simplicity, we use the abbreviations NM, UP and UP(B) to refer to NM-APG, UPFAG and UPFAG-BB, respectively, both in the discussions	and tables below.

	We now provide the details of our implementation of the four methods mentioned in the previous paragraph. AG was implemented as described in Algorithm 1 of \cite{nonconv_lan16} with sequences $ \{\alpha_k\} $, $ \{\beta_k\} $ and $ \{\lam_k \} $ chosen as $ (\alpha_k,\beta_k, \lam_k)=(2/(k+1),0.99/M,k\beta_k/2) $ for $ k\ge 1 $.
	NM was implemented as described in Algorithm 2 of \cite{Li_Lin2015} with the quadruple $ (\alpha_x,\alpha_y,\eta,\delta)$ chosen to be $(0.99/M,0.99/M,0.9,1) $.
	The code for UP was made available by the authors of \cite{LanUniformly} where
		UP is described (see Algorithm 1 of \cite{LanUniformly}). In particular, we have used their choice of parameters
		but have modified the code slightly to accommodate for the termination criterion \eqref{rho-sol} used in our benchmark.
		More specifically, the parameters $ (\hat \lam_0, \hat \beta_0, \gamma_1,\gamma_2,\gamma_3,\delta)$  needed as input by UP
		were set to
		$ (1/\bar M,1/\bar M,0.4,0.4,1,10^{-3}) $.
		UP(B) also requires the same parameters as UP and an additional one denoted
		by $\sigma$ in \cite{LanUniformly} which were set to the same values used in UP and to
		 $ \sigma=10^{-10} $, respectively.

	It is worth making
	the following remarks about the above method:
	i) AG and NM require two resolvent evaluations of $ h$ per iteration while  NC requires only one (see the third remark after NC);
	ii) NM reduces to the composite gradient method when a certain descent property is not satisfied;
	iii) AD, AD(B), RA, RA(B), UP and UP(B)
	can work without the knowledge of a curvature pair $ (M,m) $;
	and iv) UP and UP(B) adaptively compute both accelerated steps and unaccelerated ones using line searches.


	We implement all methods in MATLAB 2017b scripts and run them on a MacBook Pro with a 4-core Intel Core i7 processor and 16 GB of memory.
	

	\subsection{Nonconvex quadratic programming problem}\label{subsec:vector}
	This subsection discusses the performance of NC and its adaptive variants to solve the same quadratic programming problem as in \cite{KongMeloMonteiro,jliang2018double}, namely:
	\raggedbottom
	\begin{equation}\label{testQPprob}
		\min\left\{ f(z):=-\frac{\alpha_1}{2}\|DBz\|^{2}+\frac{\alpha_2}{2}\|Az-b\|^{2}:z\in\Delta_{n}\right\},
	\end{equation}
	where $(\alpha_1,\alpha_2)\in \R^2_{++}$, $D\in \R^{n\times n}$ is a diagonal matrix with diagonal
	entries sampled from the discrete uniform distribution ${\cal U}\{1,1000\}$,
	matrices $A\in \R^{l\times n}$,  $B\in \R^{n\times n}$ and vector $b\in \R^{ l}$ are such that their entries are generated
	from the uniform distribution ${\cal U}[0,1]$, and
	$\Delta_{n}:=\left\{ z\in\R^n:\sum_{i=1}^{n}z_{i}=1, \;\; z_i\geq0\right\}$ is the $ (n-1) $-dimensional standard simplex.
	The dimensions are set to be $(l,n)=(20,1200)$.  
	For some chosen curvature pairs $(\bar m, \bar M)\in\R^2_{++}$, the scalars $\alpha_1$ and $\alpha_2$ were chosen so that
	$\bar M=\lambda_{\max}(\nabla^{2}f)$ and $- \bar m=\lambda_{\min}(\nabla^{2}f)$ where $\lambda_{\max}(\cdot)$ and $\lambda_{\min}(\cdot)$
	denote the largest and smallest eigenvalues functions, respectively. Note that we set $ \Omega=\R^n $ in this subsection.
	
	In addition to the nine methods described at the beginning of Section \ref{sec:computResults},
	this subsection (and only this one) also reports the performance of a quasi-Newton
	variant of UPFAG, called QN, as described in \cite{LanUniformly} (see its paragraph containing (2.13)).
		Each iteration of QN  performs an unaccelerated step with respect to a variable metric
		and whose computation requires the
		evaluation of a point-to-point
		operator of the form
		$(I+ V^{-1} \partial h)^{-1}(\cdot)$ for some $V\in S_{++}^n $ (see \cite{becker2012quasi}).
		More specifically, QN is almost the same as UP (and hence has the same set of parameters as UP), except that it replaces (2.10) by (2.13) in \cite{LanUniformly}, where the quasi-Newton matrix $ G_k $ in (2.13) is updated as in the symmetric-rank-1 method (see \cite{becker2012quasi}).
	
	In our implementation, all methods use the centroid of $\Delta_{n}$ as the initial point $z_0$ and terminate with a pair $(z,v)$ satisfying
	\begin{equation}\label{eq:comp_term}
		v\in \nabla f(z)+N_{\Delta_n}(z), \qquad  \frac{\|v\|}{\|\nabla f(z_{0})\|+1}\leq 10^{-7}.
	\end{equation}
	The input triple of NC is set to 
	$ (M,m,A_0)=(\bar M/0.99,\bar m,1000) $ and
	that of AD  is set to 
	$ (M_0,m_0,\theta)=(1,1,1.25) $.

	Test cases specified by pairs $ (\bar M,\bar m) $ are generated by choosing the corresponding $\alpha_1$ and $ \alpha_2 $ as discussed in the first paragraph in this subsection. Computational results for ten methods with fixed $ \bar M=16777216 $ 
	are presented Table \ref{tab:t1} and with fixed $ \bar m=1 $ are presented in Table \ref{tab:t2}.
	In each table, the first column gives the values of $ \bar m $ or $ \bar M $ used to generate the instances,
	the second to eighth (resp., ninth to eleventh) columns provide the number of iterations and running times of AG, UP, QN, NM, NC, AD and RA (resp., UP(B), AD(B) and RA(B)). 
	The objective function values obtained by all methods are not reported since they are essentially the same on all instances.
	The bold numbers highlight the methods (using and without using Barzilai-Borwein stepsizes) that have the best performance for each case. 
	The numbers marked with * indicate that the maximum number of iterations has been reached.

	\begin{table}[H]
		\begin{centering}
			\begin{tabular}{|>{\centering}p{1.3cm}|>{\centering}p{0.7cm}>{\centering}p{0.7cm}>{\centering}p{0.7cm}>{\centering}p{0.9cm}>{\centering}p{0.7cm}>{\centering}p{0.7cm}>{\centering}p{0.7cm}|>{\centering}p{0.9cm}>{\centering}p{0.9cm}>{\centering}p{0.9cm}|}
				\hline 
				{\small{}$\bar m$} & \multicolumn{7}{c|}{\makecell{\small{}Iteration Count / \\ \small{}Running Time (s)}}
				& \multicolumn{3}{c|}{\makecell{\small{}Iteration Count / \\ \small{}Running Time (s)}}
				\tabularnewline  
				\cline{2-11}
				& {\small{}AG} & {\small{}UP} & {\small{}QN} & {\small{}NM} & {\small{}NC}& {\small{}AD} & {\small{}RA} &  {\small{}UP(B)} & {\small{}AD(B)} & {\small{}RA(B)}
				\tabularnewline
				\hline 
				{\small{}16777216} & {\small{}638\\97} & {\small{}220\\47} & {\small{}219\\64} & {\small{}251\\31} & {\small{}2376}\\ {\small{}286} & {\small{}3\\ \textbf{1}} & {\small{}3\\ \textbf{1}} 
				& {\small{}605\\ 258}  & {\small{}3\\ 3} & {\small{}3\\ \textbf{2}} 
				\tabularnewline \hline
				{\small{}1048576} & {\small{}1358\\224} & {\small{}1176\\252} & {\small{}103\\34}& {\small{}1157\\184} & {\small{}3469}\\{\small{}421} & {\small{}318\\63} & {\small{}58\\ \textbf{12}} 
				& {\small{}10\\ \textbf{6}} & {\small{}19\\ 9} & {\small{}17\\ \textbf{6}} 
				\tabularnewline
				\hline
				{\small{}65536} & {\small{}22293\\3524} & {\small{}5676\\1284} & {\small{}2737\\959}& {\small{}44705\\6525} & {\small{}3832}\\{459} & {\small{}747\\{157}} & {\small{}80\\ \textbf{18}}
				& {\small{}30\\16} & {\small{}57\\20} & {\small{}30\\\textbf{10}}
				\tabularnewline \hline
				{\small{}4096} & {\small{}31385\\5184} & {\small{}8286\\1918} & {\small{}919\\320} & {\small{}50000*\\7070} & {\small{}{17585}\\2101} & {\small{}1000\\{211}} & {\small{}74\\\textbf{18}}
				& {\small{}39\\{21}} & {\small{}90\\{34}} & {\small{}36\\\textbf{14}}
				\tabularnewline
				\hline
				{\small{}256} & {\small{}26961\\4369} & {\small{}7464\\1667} & {\small{}3410\\1126} & {\small{}49602\\7001} & {\small{}{31333}\\3713} & {\small{}969\\{216}} & {\small{}76\\\textbf{18}}
				& {\small{}35\\{18}} & {\small{}95\\{34}} & {\small{}44\\\textbf{17}}
				\tabularnewline
				\hline
				{\small{}16} & {\small{}26918\\4215} & {\small{}7334\\1609} & {\small{}665\\221} & {\small{}49515\\6806} & {\small{}{32517}\\3958} & {\small{}967\\{223}} & {\small{}75\\\textbf{18}}
				& {\small{}30\\{15}} & {\small{}80\\{29}} & {\small{}34\\\textbf{13}}
				\tabularnewline
				\hline
			\end{tabular}
			\par\end{centering}
		\caption{Numerical results for instances with fixed $ \bar M=16777216 $}\label{tab:t1}
	\end{table}
	\vspace{-5mm}

	\begin{table}[H]
		\begin{centering}
			\begin{tabular}{|>{\centering}p{1.3cm}|>{\centering}p{0.7cm}>{\centering}p{0.7cm}>{\centering}p{0.9cm}>{\centering}p{0.9cm}>{\centering}p{0.7cm}>{\centering}p{0.7cm}>{\centering}p{0.7cm}|>{\centering}p{0.9cm}>{\centering}p{0.9cm}>{\centering}p{0.9cm}|}
				\hline 
				{\small{}$\bar M$} & \multicolumn{7}{c|}{\makecell{\small{}Iteration Count / \\ \small{}Running Time (s)}}
				& \multicolumn{3}{c|}{\makecell{\small{}Iteration Count / \\ \small{}Running Time (s)}}
				\tabularnewline  
				\cline{2-11}
				& {\small{}AG} & {\small{}UP} & {\small{}QN} & {\small{}NM} & {\small{}NC}& {\small{}AD} & {\small{}RA} &  {\small{}UP(B)} & {\small{}AD(B)} & {\small{}RA(B)}
				\tabularnewline
				\hline 
				{\small{}4000} & {\small{}31403\\5284} & {\small{}7857\\1682} & {\small{}50000*\\16214} & {\small{}50000*\\7270} & {\small{}17577}\\ {\small{}2244} & {\small{}244\\50} & {\small{}105\\ \textbf{20}} 
				& {\small{}43\\ \textbf{15}}  & {\small{}58\\ 18} & {\small{}58\\ 17} 
				\tabularnewline \hline
				{\small{}16000} & {\small{}20193\\3504} & {\small{}7857\\1739} & {\small{}50000*\\14850}& {\small{}50000*\\7884} & {\small{}30239}\\{\small{}3638} & {\small{}472\\105} & {\small{}79\\ \textbf{18}} 
				& {\small{}35\\ 15} & {\small{}51\\ 18} & {\small{}34\\ \textbf{12}} 
				\tabularnewline
				\hline
				{\small{}64000} & {\small{}26962\\4891} & {\small{}7464\\1652} & {\small{}50000*\\15511}& {\small{}49592\\7628} & {\small{}31334}\\{3803} & {\small{}560\\{125}} & {\small{}77\\ \textbf{18}}
				& {\small{}38\\16} & {\small{}64\\23} & {\small{}37\\ \textbf{13}}
				\tabularnewline \hline
				{\small{}256000} & {\small{}26926\\4759} & {\small{}7364\\1522} & {\small{}3488\\1131} & {\small{}49534\\7541} & {\small{}{32527}\\3980} & {\small{}930\\{206}} & {\small{}75\\\textbf{18}}
				& {\small{}38\\{20}} & {\small{}72\\{27}} & {\small{}36\\\textbf{14}}
				\tabularnewline
				\hline
				{\small{}1024000} & {\small{}26918\\4717} & {\small{}7364\\1601} & {\small{}3234\\1028} & {\small{}49521\\7815} & {\small{}{32518}\\4092} & {\small{}967\\{227}} & {\small{}74\\\textbf{18}}
				& {\small{}38\\{22}} & {\small{}77\\{29}} & {\small{}35\\\textbf{13}}
				\tabularnewline
				\hline
				{\small{}4096000} & {\small{}26916\\4547} & {\small{}7264\\1602} & {\small{}99\\33} & {\small{}49523\\7847} & {\small{}{32515}\\4265} & {\small{}967\\{231}} & {\small{}79\\\textbf{18}}
				& {\small{}39\\{21}} & {\small{}82\\{32}} & {\small{}36\\\textbf{13}}
				\tabularnewline
				\hline
			\end{tabular}
			\par\end{centering}
		\caption{Numerical results for instances with fixed $ \bar m=1 $}\label{tab:t2}
	\end{table}
	
	
	In summary, computational results demonstrate that: i) among the methods which do not use the Barzilai-Borwein stepsize (see columns 2-8 of Tables \ref{tab:t1}-\ref{tab:t2}), RA has the best performance in terms of running time; ii) UP(B) is comparable with RA (see columns 8 and 9 of Tables \ref{tab:t1}-\ref{tab:t2}); and iii) RA(B) has the best performance among the three methods which use the Barzilai-Borwein stepsize (see columns 9-11 of Tables \ref{tab:t1}-\ref{tab:t2}).

	\subsection{Matrix problem} \label{subsec:matrix}
	In this subsection, we test our methods on a matrix version of the nonconvex quadratic programming problem
	\[
	\min\left\{ f(Z):=-\frac{\alpha_1}{2}\|D\mathcal{B}(Z)\|^{2}+\frac{\alpha_2}{2}\|\mathcal{A}(Z)-b\|^{2}:Z\in P_{n}\right\},
	\]
	where $ \mathcal{A}:S_+^n\rightarrow \R^l $ and $ \mathcal{B}:S_+^n\rightarrow \R^n $ are linear operators defined by
	\begin{align*}
		&\left[ \mathcal{A}(Z) \right]_i=\inner{A_i}{Z}_F \text{ for } A_i\in \R^{n\times n} \text{ and } 1\le i\le l,\\
		&\left[ \mathcal{B}(Z) \right]_j=\inner{B_j}{Z}_F \text{ for } B_j\in \R^{n\times n} \text{ and } 1\le j\le n,
	\end{align*}
	with entries of $ A_i $, $ B_j $  sampled
	from the uniform distribution ${\cal U}[0,1]$, 
	and $ P_n $ denotes the spectraplex
	\[
	P_n:=\{Z\in S_+^n:\text{tr}(Z)=1\}.
	\]
	$(\alpha_1,\alpha_2)$, $ D $ and $ b $ are defined as those in Subsection \ref{subsec:vector}. Note that we set $ \Omega= S_+^n $ in this subsection.
	
	
	All methods used the centroid of $ P_n $ as the initial point $ Z_0 $, i.e., $ Z_0 = I_n/n $, where $ I_n $ is the identity matrix of size $ n\times n $. Termination criterion is the same as \eqref{eq:comp_term} except that $ \Delta_n $ is replaced by $ P_n $.
	The input triple of NC is set to 
	 $ (M,m,A_0)=(\bar M/0.99,\bar m,1000) $ and
	that of AD  is set to 
	$ (M_0,m_0,\theta)=(1,1000,1.25) $.
	
	Test cases specified by pairs $ (\bar M,\bar m) $ are generated by choosing the corresponding $\alpha_1$ and $ \alpha_2 $ as discussed in the first paragraph in this subsection. Computational results of all methods with fixed $ \bar M=1000000 $ are presented in Tables \ref{tab:t3}-\ref{tab:t5}. Their formats are the same as that of Table \ref{tab:t1}. 
	The objective function values obtained by all methods are not reported since they are essentially the same on all instances.
	The bold numbers highlight the methods (using and without using Barzilai-Borwein stepsizes) that have the best performance for each case.

	\begin{table}[H]
		\begin{centering}
			\begin{tabular}{|>{\centering}p{1.3cm}|>{\centering}p{0.7cm}>{\centering}p{0.7cm}>{\centering}p{0.8cm}>{\centering}p{0.7cm}>{\centering}p{0.7cm}>{\centering}p{0.7cm}|>{\centering}p{0.9cm}>{\centering}p{0.9cm}>{\centering}p{0.9cm}|}
				\hline 
				{\small{}$\bar m$} & \multicolumn{6}{c|}{\makecell{\small{}Iteration Count / \\ \small{}Running Time (s)}}
				& \multicolumn{3}{c|}{\makecell{\small{}Iteration Count / \\ \small{}Running Time (s)}}
				\tabularnewline  
				\cline{2-10}
				& {\small{}AG} & {\small{}UP} & {\small{}NM} & {\small{}NC}& {\small{}AD} & {\small{}RA} &  {\small{}UP(B)} & {\small{}AD(B)} & {\small{}RA(B)}
				\tabularnewline
				\hline 
				{\small{}1000000} & {\small{}46\\2} & {\small{}12\\\textbf{1}} & {\small{}80\\2} & {\small{}33}\\ {\small{}\textbf{1}} & {\small{}12\\\textbf{1}} & {\small{}12\\ \textbf{1}} 
				& {\small{}9\\ \textbf{1}}  & {\small{}11\\ \textbf{1}} & {\small{}12\\ \textbf{1}} 
				\tabularnewline \hline
				{\small{}100000} & {\small{}3809\\138} & {\small{}2577\\113}& {\small{}6242\\191} & {\small{}3960}\\{\small{}94} & {\small{}2206\\87} & {\small{}597\\ \textbf{25}} 
				& {\small{}2573\\ 274} & {\small{}593\\ 41} & {\small{}282\\ \textbf{21}} 
				\tabularnewline
				\hline
				{\small{}10000} & {\small{}5400\\198} & {\small{}7697\\347}& {\small{}10404\\328} & {\small{}1247}\\\textbf{29} & {\small{}2591\\{103}} & {\small{}1290\\ {54}}
				& {\small{}6811\\671} & {\small{}835\\57} & {\small{}569\\ \textbf{40}}
				\tabularnewline \hline
				{\small{}1000} & {\small{}4621\\163} & {\small{}6759\\308} & {\small{}11053\\360} & {\small{}{4424}\\111} & {\small{}2637\\{104}} & {\small{}1211\\\textbf{51}}
				& {\small{}6384\\{646}} & {\small{}721\\{48}} & {\small{}581\\\textbf{41}}
				\tabularnewline
				\hline
				{\small{}100} & {\small{}4476\\157} & {\small{}6620\\299} & {\small{}11271\\312} & {\small{}{8870}\\218} & {\small{}2639\\{113}} & {\small{}1373\\\textbf{57}}
				& {\small{}6876\\{683}} & {\small{}812\\{54}} & {\small{}535\\\textbf{37}}
				\tabularnewline
				\hline
			\end{tabular}
			\par\end{centering}
		\caption{Numerical results for instances with fixed $ \bar M=1000000 $}\label{tab:t3}
	\end{table}
	\vspace{-5mm}
	
	In Table \ref{tab:t3}, the dimensions are set to be $ (l, n) = (50, 200)$ and $ 2.5\% $ of entries in $ A_i, B_j  $ are nonzero.

	\begin{table}[H]
		\begin{centering}
			\begin{tabular}{|>{\centering}p{1.3cm}|>{\centering}p{0.7cm}>{\centering}p{0.7cm}>{\centering}p{0.8cm}>{\centering}p{0.7cm}>{\centering}p{0.7cm}>{\centering}p{0.7cm}|>{\centering}p{0.9cm}>{\centering}p{0.9cm}>{\centering}p{0.9cm}|}
				\hline 
				{\small{}$\bar m$} & \multicolumn{6}{c|}{\makecell{\small{}Iteration Count / \\ \small{}Running Time (s)}}
				& \multicolumn{3}{c|}{\makecell{\small{}Iteration Count / \\ \small{}Running Time (s)}}
				\tabularnewline  
				\cline{2-10}
				& {\small{}AG} & {\small{}UP} & {\small{}NM} & {\small{}NC}& {\small{}AD} & {\small{}RA} &  {\small{}UP(B)} & {\small{}AD(B)} & {\small{}RA(B)}
				\tabularnewline
				\hline 
				{\small{}1000000} & {\small{}44\\4} & {\small{}12\\\textbf{1}} & {\small{}75\\5} & {\small{}32}\\ {\small{}2} & {\small{}12\\2} & {\small{}12\\ 2} 
				& {\small{}10\\ \textbf{2}}  & {\small{}12\\ \textbf{2}} & {\small{}12\\ \textbf{2}} 
				\tabularnewline \hline
				{\small{}100000} & {\small{}1411\\134} & {\small{}621\\69}& {\small{}3151\\224} & {\small{}635}\\{\small{}40} & {\small{}530\\52} & {\small{}240\\ \textbf{25}} 
				& {\small{}57\\ 13} & {\small{}151\\ 28} & {\small{}61\\ \textbf{11}} 
				\tabularnewline
				\hline
				{\small{}10000} & {\small{}1963\\195} & {\small{}1733\\191}& {\small{}5071\\373} & {\small{}1104}\\{69} & {\small{}868\\{86}} & {\small{}198\\ \textbf{21}}
				& {\small{}109\\31} & {\small{}211\\39} & {\small{}137\\ \textbf{25}}
				\tabularnewline \hline
				{\small{}1000} & {\small{}1935\\193} & {\small{}1792\\197} & {\small{}5172\\382} & {\small{}{3823}\\244} & {\small{}900\\{94}} & {\small{}215\\\textbf{23}}
				& {\small{}97\\\textbf{25}} & {\small{}208\\{38}} & {\small{}160\\{29}}
				\tabularnewline
				\hline
				{\small{}100} & {\small{}1934\\190} & {\small{}1803\\197} & {\small{}5045\\367} & {\small{}{5771}\\391} & {\small{}904\\{95}} & {\small{}210\\\textbf{23}}
				& {\small{}112\\{29}} & {\small{}225\\{40}} & {\small{}147\\\textbf{27}}
				\tabularnewline
				\hline
			\end{tabular}
			\par\end{centering}
		\caption{Numerical results for instances with fixed $ \bar M=1000000 $}\label{tab:t4}
	\end{table}
	\vspace{-5mm}
	
	In Table \ref{tab:t4}, the dimensions are set to be $ (l, n) = (50, 400)$ and $ 0.5\% $ of entries in $ A_i, B_j  $ are nonzero.

	\begin{table}[H]
		\begin{centering}
			\begin{tabular}{|>{\centering}p{1.3cm}|>{\centering}p{0.7cm}>{\centering}p{0.7cm}>{\centering}p{0.8cm}>{\centering}p{0.7cm}>{\centering}p{0.7cm}>{\centering}p{0.7cm}|>{\centering}p{0.9cm}>{\centering}p{0.9cm}>{\centering}p{0.9cm}|}
				\hline 
				{\small{}$\bar m$} & \multicolumn{6}{c|}{\makecell{\small{}Iteration Count / \\ \small{}Running Time (s)}}
				& \multicolumn{3}{c|}{\makecell{\small{}Iteration Count / \\ \small{}Running Time (s)}}
				\tabularnewline  
				\cline{2-10}
				& {\small{}AG} & {\small{}UP} & {\small{}NM} & {\small{}NC}& {\small{}AD} & {\small{}RA} &  {\small{}UP(B)} & {\small{}AD(B)} & {\small{}RA(B)}
				\tabularnewline
				\hline 
				{\small{}1000000} & {\small{}69\\22} & {\small{}16\\6} & {\small{}117\\26} & {\small{}39}\\ {\small{}8} & {\small{}11\\\textbf{5}} & {\small{}11\\ 6} 
				& {\small{}13\\ 8}  & {\small{}11\\ \textbf{7}} & {\small{}11\\ \textbf{7}} 
				\tabularnewline \hline
				{\small{}100000} & {\small{}277\\119} & {\small{}58\\21}& {\small{}502\\118} & {\small{}165}\\{\small{}39} & {\small{}24\\10} & {\small{}8\\ \textbf{3}} 
				& {\small{}9\\ 7} & {\small{}8\\ \textbf{4}} & {\small{}8\\ \textbf{4}} 
				\tabularnewline
				\hline
				{\small{}10000} & {\small{}491\\173} & {\small{}141\\52}& {\small{}1030\\246} & {\small{}703}\\{168} & {\small{}60\\{23}} & {\small{}60\\ \textbf{21}}
				& {\small{}13\\10} & {\small{}13\\\textbf{7}} & {\small{}13\\8}
				\tabularnewline \hline
				{\small{}1000} & {\small{}531\\169} & {\small{}161\\60} & {\small{}1144\\259} & {\small{}{1326}\\309} & {\small{}70\\{26}} & {\small{}70\\\textbf{25}}
				& {\small{}13\\{10}} & {\small{}15\\\textbf{9}} & {\small{}15\\\textbf{9}}
				\tabularnewline
				\hline
				{\small{}100} & {\small{}535\\172} & {\small{}163\\61} & {\small{}1156\\260} & {\small{}{1482}\\336} & {\small{}71\\{26}} & {\small{}71\\\textbf{25}}
				& {\small{}13\\\textbf{10}} & {\small{}16\\\textbf{10}} & {\small{}16\\\textbf{10}}
				\tabularnewline
				\hline
			\end{tabular}
			\par\end{centering}
		\caption{Numerical results for instances with fixed $ \bar M=1000000 $}\label{tab:t5}
	\end{table}
	\vspace{-5mm}
	
	In Table \ref{tab:t5}, the dimensions are set to be $ (l, n) = (50, 800)$ and $ 0.1\% $ of entries in $ A_i, B_j  $ are nonzero.

	In summary, computational results demonstrate that: i) among the methods which do not use the Barzilai-Borwein stepsize (see columns 2-7 of Tables \ref{tab:t3}-\ref{tab:t5}), RA has the best performance in terms of running time; ii) UP(B) is comparable with RA in many instances (see columns 7 and 8 of Tables \ref{tab:t3}-\ref{tab:t5}); and iii) RA(B) has the best performance among the three methods which use the Barzilai-Borwein stepsize (see columns 8-10 of Tables \ref{tab:t3}-\ref{tab:t5}). 
	
	\subsection{Matrix completion} \label{subsec:MC}
	
	This subsection focuses on a constrained version of the nonconvex low-rank matrix completion problem
	studied in \cite{jliang2019average,yao2017efficient}.
	
	Given an incomplete observed matrix $ O $ with the set $ {\cal Q} $ of observed entries, parameters $\beta>0$ and $\tau>0$ and letting $p:\R\to\R_+ $ denote the log-sum penalty 
	\[
	p(t)= p_{\beta,\tau}(t) := \beta\log\left( 1+\frac{|t|}{\tau}\right)
	\]
	and $ \Pi_{\cal Q} $ denote the linear operator that maps a matrix $ A $ to the matrix whose entries
	in ${\cal Q}$ have the same values of the corresponding ones in $A$ and
	whose entries outside of ${\cal Q}$ are all zero,
	then the constrained version of the matrix completion problem is formulated as
	\begin{equation}\label{eq:MC struc}
		\min_{X\in\R^{l\times n}} f(X)+h(X),
	\end{equation}
	where
	\begin{align*}
		f(X)=\frac12 \|\Pi_{\cal Q}(X-O)\|_F^2+ \mu\sum_{i=1}^{r}[p(\sigma_i(X))-p_0\sigma_i(X)], \\
		h(X)=\mu p_0\|X\|_* + I_{{\cal B}(R)}(X), \quad p_0=p'(0)=\frac{\beta}{\tau},
	\end{align*}
	$R$ is a positive scalar, ${\cal B}(R) := \{ X \in \R^{l \times n} : \|X\|_F \le R\}$,
	$O \in \R^{\cal Q}$ is an incomplete observed matrix, $ \mu>0 $ is a parameter,
	$ r := \min \{l,n\}$ and $\sigma_i(X)$ is the $ i $-th singular value of $ X $
	and $ \|\cdot\|_*$ denotes the nuclear norm defined as
	$ \|\cdot\|_* := \sum_{i=1}^r \sigma_i(\cdot)$. Note that we set $ \Omega=\R^{l\times n} $ in this subsection.
	It is shown in \cite{jliang2019average,yao2017efficient} that the problem in \eqref{eq:MC struc} falls into the general class of SNCO problems,  
	\[
	f(X')-f(X)- \inner{\nabla f(X')}{X'-X}_F\le \frac{\tilde M}{2}\|X'-X\|_F^2, \quad \forall X, X' \in \Omega
	\]
	for $ \tilde M=1 $
	and that the pair
	\begin{equation}\label{eq:M-MC}
		(\bar M,\bar m)=\left( \max\left\lbrace \tilde M,\frac{2\mu\beta}{\tau^2}\right\rbrace ,\frac{2\mu\beta}{\tau^2} \right) 
	\end{equation}
	satisfies \eqref{ineq:Lips} and \eqref{ineq:curv}.
	
	We use the {\it MovieLens} dataset\footnote{http://grouplens.org/datasets/movielens/}
	to obtain the observed index set ${\cal Q}$ and the incomplete observed matrix $ O $.
	The dataset includes a sparse matrix with 100,000 ratings of \{1,2,3,4,5\} from 943 users on 1682 movies.
	The radius $ R $ is chosen as the Frobenius norm of the matrix of size $ 943 \times 1682 $
	containing the same entries as $ O $ in $ {\cal Q} $ and 5 in the entries outside of $ {\cal Q} $.
	
	All methods take a random matrix $ Z_0 $ sampled from the standard Gaussian distribution as the initial point, where the random number generation seed is fixed, and terminates with a pair $ (Z,V) $ satisfying
	\[
	V\in \nabla f(Z)+\partial h(Z), \qquad  \frac{\|V\|_F}{\|\nabla f(Z_{0})\|_F+1}\leq 5\times 10^{-4}.
	\]
	The input triple of NC is set to 
	 $ (M,m,A_0)=(\tilde M,\tilde M,2) $, since $ \tilde M $ is the one actually needed in the convergence analysis of this algorithm (see Lemma \ref{consequenceproposition}). 
	The input triple of AD  is set to 
	$ (M_0,m_0,\theta)=(1,0.5,1.25) $.
	
	Computational results of all methods are summarized in Table \ref{tab:t6}. Specifically, the first column
	gives the values of $\bar M$ computed according to \eqref{eq:M-MC} with four different triples $ (\mu, \beta, \tau) $,
	the second to seventh columns provide the function values of \eqref{eq:MC struc} at the last iteration and the number of iterations, and the eighth to thirteenth columns present the running times. The bold numbers highlight the methods that have the best performance for each case.
	The results of RA are not reported since they are the same as those of AD, which is due to the fact that $ \{\phi(y_k) \} $ generated by AD is a decreasing sequence and hence no restart is performed in RA.

	\begin{table}[H]
		\begin{centering}
			\begin{tabular}{|>{\centering}p{0.4cm}|>{\centering}p{1cm}>{\centering}p{1cm}>{\centering}p{1cm}>{\centering}p{1cm}>{\centering}p{1cm}>{\centering}p{1cm}|>{\centering}p{0.7cm}>{\centering}p{0.7cm}>{\centering}p{1cm}>{\centering}p{0.7cm}>{\centering}p{0.7cm}>{\centering}p{0.7cm}|}
				\hline 
				{\small{}$\bar M$} & \multicolumn{6}{c|}{\makecell{\small{}Function Value / \\ \small{}Iteration Count}}
				& \multicolumn{6}{c|}{\makecell{\small{}Running Time (s)}}
				\tabularnewline  
				\cline{2-13}
				& {\small{}AG} & {\small{}UP} & {\small{}UP(B)} & {\small{}NM} & {\small{}NC}& {\small{}AD} & {\small{}AG} & {\small{}UP} & {\small{}UP(B)} & {\small{}NM} & {\small{}NC} & {\small{}AD}
				\tabularnewline
				\hline 
				{\small{}4.4} & {\small{}2257\\3856} & {\small{}2670\\898} & {\small{}2605\\521} & \textbf{\small{}1809}\\ {\small{}1036} & {\small{}2605\\1491} & {\small{}2625\\ 1219} 
				& {\small{}4568} & {\small{}2214} & {\small{}1545} & {\small{}1033} & {\small{}1114} & \textbf{\small{}1021} 
				\tabularnewline \hline
				{\small{}8.9} & {\small{}3886\\9158} & {\small{}4322\\1782} & {\small{}4261\\576} & \textbf{\small{}3359}\\{\small{}1617} & {\small{}4154\\1642} & {\small{}4203\\ 1302} 
				& {\small{}10251} & {\small{}2592} & {\small{}1621} & {\small{}1605} & {\small{}1202} & \textbf{\small{}1089} 
				\tabularnewline
				\hline
				{\small{}20} & {\small{}4282\\22902} & {\small{}4736\\3962} & {\small{}4637\\898} & \textbf{\small{}3635}\\{2875} & {\small{}4637\\{676}} & {\small{}4582\\2177}
				& {\small{}29274} & {\small{}5850} & {\small{}1914} & {\small{}2836} & \textbf{\small{}1178} & {\small{}1822}
				\tabularnewline \hline
				{\small{}30} & {\small{}5967\\37032} & {\small{}6475\\5857} & {\small{}6753\\606} & {\small{}\textbf{5237}\\3717} & {\small{}6292\\{1646}} & {\small{}6293\\{1952}}
				& {\small{}41673} & {\small{}8159} & {\small{}1628} & {\small{}4182} & \textbf{\small{}1233} & {\small{}1633}
				\tabularnewline
				\hline
			\end{tabular}
			\par\end{centering}
		\caption{Numerical results for matrix completion instances}\label{tab:t6}
	\end{table}
	\vspace{-5mm}
	
	
	
	In summary, computational results demonstrate that: i) NM always finds the smallest function values, since it requires objective function values to satisfy a descent property, and if violated, a projected gradient step is taken to ensure the descent in function values;
		ii) NC and AD have the best performance in terms of the running time;
		and iii) since NC and AD are good enough compared with UP(B), we do not presents the results of AD(B) and RA(B).
	
	\subsection{Nonnegative matrix factorization}\label{subsec:NMF}
	In this subsection, we further test AD on a real life application rather than artificially generated problems and data.
	NMF is a popular dimension reduction method in which a data matrix $ X $ is factored into two matrices $ V $ and $ W $, with constraints that each entry in $ V $ and $ W $ is nonnegative.
	\begin{equation}\label{eq:NMF}
		\min \left \{f(V,W):=\frac{1}{2}\|X-VW\|_F^2:V\ge0,W\ge0\right\},
	\end{equation}
	where $ X\in \R^{n\times l} $, $ V\in \R^{n\times k} $ and $ W\in \R^{k\times l} $. Note that we set $ \Omega=\R^{n\times l} $ in this subsection.
	Intuitively, the data matrix $ X $ is a collection of $ m $ data points in $ \R^n $, the columns of $ V $ can be viewed as the basis of all data points, and hence each data point is a linear combination of the basis, with weights in the corresponding column in $ W $. 
	Because of its ability of extracting easily interpretable factors and automatically performing clustering, NMF finds a wide range of applications in practice, from text mining to image processing.
	Most of the NMF algorithms solve \eqref{eq:NMF} in a two-block coordinate descent manner, by alternatively minimizing with respect to one of the two blocks, $ V $ or $ W $, while keeping the other one fixed. Alternating minimization is a natural idea for NMF, since the subproblem in one block is convex. 
	
	In this subsection, we apply AD to solve the nonconvex problem \eqref{eq:NMF} directly by minimizing in $ (V,W) $ jointly. 
	
	For a preliminary computational test, we apply AD to facial feature extraction. The problem is as described in \eqref{eq:NMF}, to factor out a data matrix into two matrices. The facial image dataset is provided by AT\&T Laboratories Cambridge \footnote{https://www.cl.cam.ac.uk/research/dtg/attarchive/facedatabase.html}. There are ten different images of each of 40 distinct subjects, and each image contains $ 92\times112 $ pixels, with 256 gray levels per pixel. It results in a matrix of size $ 10, 304\times400 $, where each column of the data matrix is the vectorization of an image.
	
	It is hard to estimate $ M $ in \eqref{ineq:Lips} due the unboundedness in NMF, so we can only apply AD, which has the benefit of working without the knowledge of $ M $. AD is benchmarked against the ANLS (Alternating Nonnegative Least Squares) method \cite{kim2008toward}. ANLS alternatively solves minimization subproblems in $ V $ and $ W $ with nonnegative constraints and the other variable being fixed. We use the implementation of ANLS \footnote{https://www.cc.gatech.edu/~hpark/nmfsoftware.html} provided by the authors of \cite{kim2008toward} as a benchmark for comparison. The ANLS code is slightly modified to
	accommodate for the termination criterion \eqref{eq:stop}.
	
	Both methods use the initial point $ (V_0,W_0)= (\mathbbm{1}^{n\times k}/(nk), \mathbbm{1}^{k\times l}/(kl)) $, where $ \mathbbm{1}^{n\times k} $ and $ \mathbbm{1}^{k\times l} $ are all one matrices of size $ n\times k $ and $ k\times l $. $ k $ is set to be 20. AD terminates with a pair $((V,W),(S_V,S_W))$ satisfying
	\begin{equation}\label{eq:stop}
		(S_V,S_W) \in \nabla f(V,W)+N_{\mathcal{F}}(V,W), \qquad  \frac{\|(S_V,S_W)\|_F}{\|\nabla f(V_0,W_0)\|_F+1}\leq 10^{-7},
	\end{equation}
	where $ \mathcal{F} = \{(V,W)\in \R^{n\times k} \times \R^{k\times l}:V\ge 0, W\ge 0\} $.
	The input triple of AD  is set to $(M_0, m_0, \theta)=(1, 1000, 1.25)$. Computational results are summarized in Table \ref{tab:t7}.
	
	\begin{table}[H]
		\begin{centering}
			\begin{tabular}{|>{\centering}p{2cm}|>{\centering}p{3cm}|>{\centering}p{3cm}|>{\centering}p{3cm}|}
				\hline 
				{Method} & {Function Value} & {Iteration Count}& {Running Time(s)}\tabularnewline
				\hline 
				{\small\text{AD}} & {\small{}2.80E+09} & {\small{}28} &  {\small{}4.6}\tabularnewline
				{\small\text{ANLS}}  & {\small{}1.20E+09} & {\small{}1000*} & {\small{}137.6} \tabularnewline
				\hline  
			\end{tabular}
			\par\end{centering}
		\caption{Numerical results for NMF}\label{tab:t7}
	\end{table}
	

	In summary, computational results demonstrate that ANLS reaches the maximum number of iterations (i.e., 1000), and AD outperforms ANLS in terms of the running time.

	\section{Concluding remarks}\label{sec:conclusions}
	
	This paper presents two ACG variants and establishes their iteration-complexities for obtaining an approximate solution of the SNCO problem. Numerical results are also given showing that they are both efficient in practice.
	
	We have not assumed in our analysis that the set $\Omega$ as in assumption (A3) is bounded. However,
	we remark that if $\Omega$ is bounded then it can be shown using a simpler analysis than the one given in
	this paper that the version of the NC-FISTA with $m=0$ and $\lam=1/(2M)$
	has an
	\[
	\mathcal{O}\left( \left(\frac{M^2 d_0^2}{ \hat{\rho}^2} \right)^{1/3} + \left(\frac{M\bar m D^2_\Omega}{ \hat{\rho}^2} \right)^{1/2} + \frac{M\bar m D^2_h}{ {\hat{\rho}}^2}  +1 \right)
	\]
	iteration-complexity where 
	$D_\Omega :=\sup_{u,u' \in \Omega}  \| u'-u \| < \infty$. Moreover, it can be shown that a version of the
	ADAP-NC-FISTA in which $\lam_k$ is updated in a similar way and
	$m_k=0$ for every $k$ has a guaranteed iteration-complexity
	that lies in between the one above and the one in \eqref{cmplx}.
	
	Finally, we have implemented the two versions mentioned in the previous paragraph
	and tested them on problems for which $\Omega$ is bounded but have observed that
	they are not as efficient as the corresponding ones
	studied in this paper.

	\section{Acknowledgements}
	We are grateful to Guanghui Lan and Saeed Ghadimi for sharing the source code of the UPFAG method in \cite{LanUniformly}.
	We are also grateful to the two anonymous referees and the associate editor for their insightful comments which we have used to substantially improve the quality of this work.
	
	\bibliographystyle{plain}
	\bibliography{Proxacc_ref}

	\setcounter{claim}{0}
	\renewcommand{\theclaim}{\Alph{claim}}
	
	\appendix
	\section{Supplementary results}\label{Appendix A}
	This section provides a bound on the quantity
	$ \min_{0 \leq i \leq k-1} \|y_{i+1}-\tx_i\|^2 $ for the case in which the parameter
	$m_0$ of the ADAP-NC-FISTA satisfies
	$ m_0\ge\bar m $.
	Note that an alternative bound on this quantity has already been developed in
	Proposition \ref{lem:xi} for any $m_0>0$.
	
	\begin{proposition}\label{lem:xi2}
		For every $ k\ge 1 $, for $ m_0\ge \bar m $, we have
		\[
		\frac{1}{10}\left( \sum_{i=0}^{k-1}A_{i+1}\right) \min_{0 \leq i \leq k-1} \| y_{i+1} - \tilde{x}_i \|^2 \le 
		2\lam_0 A_0 (\phi(y_0)-\phi_*)+\|x_0-x^*\|^2+\lam_0D_h^2 \left( 2m_0+2m_0 k+ \bar m\sum_{i=0}^{k-1}a_i\right) .
		\]
	\end{proposition}
	\begin{proof}
		Using the assumption of the lemma that $ m_0\ge \bar m $, the facts that $a_i \ge 2$ for $ i\ge 0 $ from the fourth remark following SUB$ (\theta,\lam,m) $, and $ \{\lam_i\} $ is non-increasing from Lemma \ref{observation}(d), we have
		\begin{equation}\label{ineq:new}
			\left( \bar m+\frac{2m_0}{a_i}\right) \lam_{i+1}\le m_0 \left( 1  +\frac{2}{a_i}\right) \lam_{i+1} \le 2m_0\lam_i.
		\end{equation}
		The above inequality implies that \eqref{req:xi} is always satisfied with $ m=m_0 $ and $ \lam=\lam_{i+1} $. Hence, $ m_k $ is never updated in SUB$ (\theta,\lam,m) $, i.e., $ m_i=m_0 $, for $ i\ge0 $.
		Using similar arguments as in the proof of Lemma \ref{invariantinequality}, we conclude that  for every $i \ge 0$ and $u \in \Omega$,
		\begin{align}
			& 2\lam_{i+1} A_{i+1} \phi(y_{i+1}) +    (2m_0\lam_{i+1}+1) \| u - x_{i+1} \|^2 + (1-\lam_{i+1} {\cal C}_{i+1}) A_{i+1} \| y_{i+1} - \tilde{x}_i \|^2  \nn \\
			& \leq  2\lam_{i+1} A_i \gamma_i(y_i) + 2\lam_{i+1}  a_i \gamma_i(u) + \|u - x_i \|^2, \label{ineq:recursive2}
		\end{align}
		where
		\[
		\gamma_i(u) := \tilde\gamma_i(y_{i+1}) +  \frac{1}{\lam_{i+1}}\langle \tilde{x}_i - y_{i+1}, u - y_{i+1} \rangle + \frac{m_0}{a_i} \| u - y_{i+1} \|^2
		\]
		and
		\begin{equation}\label{def:tgamma1}
			\tilde \gamma_i(u) :=  \ell_f(u;\tilde{x}_i) + h(u) + \frac{m_0}{a_i} \| u - \tilde{x}_i \|^2.
		\end{equation}
		As in Lemma \ref{gammatildegamma}(a), we have $ \gamma_i(u)\le \tilde \gamma_i(u) $ for every $ u\in \dom h $.
		Hence, it follows from \eqref{def:tgamma1} and \eqref{ineq:curv} that for every $ k\ge 0 $ and $ u  \in \dom h$, we have 
		\begin{align}
			\gamma_i(u)-\phi(u)&\le \tilde \gamma_i(u)-\phi(u)
			=\ell_f(u;\tx_i)-f(u)+\frac{m_0}{a_i}\|u-\tx_i\|^2 
			\le \frac12\left( \bar m+\frac{2m_0}{a_i}\right) \|u-\tx_i\|^2. \label{ineq:diff2}
		\end{align}
		Taking $ u=x^* $, and using \eqref{ineq:recursive2}, \eqref{ineq:tech}, \eqref{ineq:diff2}, \eqref{ineq:new}, Lemma \ref{observation}(c),
		and the facts that $ x_0=y_0 $, $ \lam_i\le \lam_0 $ and $ \phi(y_i)\ge\phi_* $ for $ i\ge 0 $, we conclude that for every $0 \le i \le k-1$,
		\begin{align*}
			0.1A_{i+1}& \| y_{i+1} - \tilde{x}_i \|^2 -  2\lam_{i} A_{i}(\phi(y_{i}) - \phi_*) -  \| x^* - x_{i} \|^2  \nn \\
			&   + 2\lam_{i+1}A_{i+1} (\phi(y_{i+1}) - \phi_*) + (2m_0\lam_{i+1}+1)\|x^* - x_{i+1} \|^2 \nonumber  \\
			&   \le  2\lam_{i+1} A_i(\gamma_i(y_i) - \phi(y_i)) + 2\lam_{i+1} a_i(\gamma_i(x^*) - \phi_* )+2(\lam_{i+1}-\lam_{i}) A_{i}(\phi(y_{i}) - \phi_*) \nonumber \\
			&  \le \lam_{i+1}\left( \bar m + \frac{2m_0}{a_i} \right)  \left( A_i\| y_i - \tx_i \|^2 + a_i\| x^* - \tx_i \|^2 \right)  \nonumber  \\
			&  \le \lam_{i+1}\left( \bar m + \frac{2m_0}{a_i} \right) \left( \| x^* - x_i \|^2 + a_i D_h^2 \right) \nonumber  \\
			&  \le 2m_0\lam_{i} \| x_i-x^* \|^2 + (\bar m a_i+2m_0)\lam_{i+1} D_h^2 \nn \\
			&  \le 2m_0\lam_{i} \| x_i-x^* \|^2 + (\bar m a_i+2m_0)\lam_0 D_h^2.
		\end{align*}
		The conclusion is obtained by rearranging terms and summing the above inequality from $ i=0 $ to $ k-1 $.
	\end{proof}
	
\end{document}